\documentclass{amsart}

\usepackage{amssymb}

\usepackage{amsmath}

\usepackage[all]{xy}
\usepackage{tikz-cd}
\usepackage{MnSymbol}

\usepackage{graphicx}               
\usepackage{bbm}
\usepackage{mathtools}
\usepackage{hyperref}

\usepackage{extarrows}

\usepackage{mathrsfs}
\usepackage{color}

\newtheorem{theorem}{Theorem}[section]
\newtheorem{lemma}[theorem]{Lemma}
\newtheorem{proposition}[theorem]{Proposition}
\newtheorem{corollary}[theorem]{Corollary}

\newtheorem{assumption}[theorem]{Assumption}
\newtheorem{question}[theorem]{Question}

\theoremstyle{definition}
\newtheorem{definition}[theorem]{Definition}

\theoremstyle{remark}
\newtheorem{remark}[theorem]{Remark}

\DeclareMathOperator{\A}{\mathbb{A}}

\DeclareMathOperator{\C}{\mathbb{C}}

\DeclareMathOperator{\N}{\mathbb{N}}

\DeclareMathOperator{\Q}{\mathbb{Q}}

\DeclareMathOperator{\R}{\mathbb{R}}

\DeclareMathOperator{\bT}{\mathbb{T}}

\DeclareMathOperator{\Z}{\mathbb{Z}}

\newcommand{\cA}{\mathcal{A}}

\newcommand{\cC}{\mathcal{C}}
\newcommand{\cD}{\mathcal{D}}
\newcommand{\cE}{\mathcal{E}}
\newcommand{\cF}{\mathcal{F}}

\newcommand{\cH}{\mathcal{H}}
\newcommand{\cI}{\mathcal{I}}

\newcommand{\cK}{\mathcal{K}}
\newcommand{\K}{\mathcal{K}}

\newcommand{\cM}{\mathcal{M}}

\newcommand{\cO}{\mathcal{O}}
\newcommand{\roi}{\mathcal{O}}

\newcommand{\cR}{\mathcal{R}}

\newcommand{\cU}{\mathcal{U}}
\newcommand{\cV}{\mathcal{V}}
\newcommand{\cW}{\mathcal{W}}

\newcommand{\ff}{\mathfrak{f}}

\newcommand{\fl}{\mathfrak{l}}

\newcommand{\fm}{\mathfrak{m}}
\newcommand{\fn}{\mathfrak{n}}

\newcommand{\fp}{\mathfrak{p}}
\newcommand{\fq}{\mathfrak{q}}
\newcommand{\fr}{\mathfrak{r}}

\newcommand{\gn}{\mathfrak{n}}
\newcommand{\gp}{\mathfrak{p}}
\newcommand{\gf}{\mathfrak{f}}
\newcommand{\gq}{\mathfrak{q}}
\newcommand{\gm}{\mathfrak{m}}
\newcommand{\gr}{\mathfrak{r}}

\newcommand{\matrixxx}[2]{\left(\begin{matrix}
#1\\
#2
\end{matrix} \right)}

\newcommand{\smallmatrixx}[4]{\big(\begin{smallmatrix}
#1 & #2\\
#3 & #4
\end{smallmatrix}\big)}
\space

\newcommand{\matrixx}[4]{\begin{pmatrix}
#1 & #2\\
#3 & #4
\end{pmatrix}}
\space

\space

\newcommand{\tensorspace}{\roi_{K,p}}

\begin{document}

\title[The Bianchi eigenvariety around non-cuspidal points]{Geometry of the Bianchi eigenvariety around non-cuspidal points and strong multiplicity-one results. %Strong multiplicity-one for Bianchi Eisenstein eigensystems and the eigenvariety around non-cuspidal points
}

%    Remove any unused author tags.

%    author one information
\author{Daniel Barrera Salazar}
\address{}
\curraddr{}
\email{}
\thanks{}

% author two information
\author{Luis Santiago Palacios}
%\address{}
%\curraddr{}
%\email{}
%\thanks{}


\date{}

\begin{abstract}
Let $K$ be an imaginary quadratic field. In this article, we study the local geometry of the Bianchi eigenvariety around non-cuspidal classical points, in particular, ordinary non-cuspidal base change points. 

To perform this study we introduce Bianchi Eisenstein eigensystems and prove strong multiplicity-one results on the cohomology of the corresponding Bianchi threefolds. We believe these results are of independent interest. 
\end{abstract}

\maketitle

\tableofcontents 

\section{Introduction}

Eigenvarieties are rigid analytic spaces that parametrize $p$-adic automorphic representations. A classical point is a point attached to a usual automorphic representation. The study of an eigenvariety around classical points has shown important arithmetic applications and has been carried out in several settings, as for example \cite{GS93}, \cite{skinner2014iwasawa}, \cite{betina2022failure}, \cite{wiles1990iwasawa}, \cite{bertolini2022heegner}, \cite{loeffler2020bloch}, \cite{boxer2021abelian}, \cite{kisin2003overconvergent}.

 It seems that in most of the literature where the geometry around classical points has been studied the underlying locally symmetric space is a Shimura variety. When there is no a Shimura variety structure, less tools are available and therefore the local geometry of the corresponding eigenvariety becomes much harder to study. The first natural setting where this happens is the context of Bianchi manifolds, which are hyperbolic three manifolds attached to Bianchi groups. The geometry of Bianchi eigenvarieties
 around classical points attached to  \emph{cuspidal} automorphic representations seems to hide deep mysteries (see for example \cite{BW21}, \cite{rawson2024computing}, \cite{serban2022finiteness}, \cite{lee2023p}) and to be connected to important arithmetic questions as explained in \cite{loeffler2021birch}.

This work is devoted to the study of the geometry of Bianchi eigenvarieties around classical points attached to  \emph{non-cuspidal} automorphic representations. To perform this study we prove multiplicity-one results in cohomology which are interesting by themselves. Their strong versions seem to be new in the literature. Thus, our work is focused on two type of results: 
\begin{itemize}
    \item strong multiplicity-one theorems, and
    \item the geometry of the Bianchi eigenvariety.
\end{itemize}

We would like to stress the opposite behavior of $p$-adic families in the Hilbert and Bianchi setting. For example, in the Hilbert setting, where the underlying totally real number field is of degree $2$ over $\Q$, the cuspidal families are $2$-dimensional and Eisenstein families are $1$-dimensional. In the Bianchi setting exactly the opposite is true. Developing this opposite behavior was the main motivation of this work. 

Furthermore, we expect analogue phenomena to the Bianchi situation would be possible to develop for other reductive groups $G$ over $\Q$ such that $G(\R)$ does not have discrete series.
 
\subsection{Strong multiplicity-one}

In his thesis, Weisinger \cite{weisinger1977some} developed a newform theory for elliptic modular Eisenstein series, proving in particular, a strong multiplicity-one theorem. This multiplicity-one theorem was later improved by Rajan \cite{rajan2000refinement}, followed by a further refinement of Thompson and Linowitz \cite{linowitz2015sign}, who showed that Eisenstein newforms are uniquely determined by their Hecke eigenvalues for any set of primes having Dirichlet density greater than $1/2$.

The newform theory developed by Weisinger, was later extended to the Hilbert modular setting by Wiles \cite{wiles1986p} and by Atwill and Linowitz \cite{atwill2013newform}. In particular, \cite{atwill2013newform} proves that Hilbert Eisenstein newforms, just like elliptic modular Eisenstein series, are uniquely determined by their Hecke eigenvalues for any set of primes having Dirichlet density greater than $1/2$.

One of our main goals is to prove analogous results in the Bianchi setting and working with the cohomology of the corresponding Bianchi threefold. We put our attention on \emph{Bianchi Eisenstein eigensystems}, which are systems of Hecke-eigenvalues appearing in the Eisenstein part of the cohomology of a Bianchi threefold. 

Let $K$ be an imaginary quadratic field. Let $\phi= (\phi_1, \phi_2)$ be a couple of Hecke characters of $K$ with $\phi_1,\phi_2$ having conductors $\gn_1$, $\gn_2$ respectively, such that $(\gn_1,\gn_2)=1$ and let $Y_{\K_1(\gn)}$ be the Bianchi threefold of level $\gn=\gn_1\gn_2$ (see \ref{ss: classical cohomology}). Let $S$ be a set of primes of $K$ not dividing $\gn$ and consider the following ideal of the Hecke algebra $\cH_{\fn}\otimes \C$ (see def. \ref{d: algebras de hecke}):
$$\gm_{\phi, S}= (\{T_\gq-\phi_1(\gq)^{-1}-N(\fq)\phi_2(\gq)^{-1},T_{\gq,\gq}-\phi_1(\gq)^{-1}\phi_2(\gq)^{-1} : \gq\in S\}).$$ 

The following theorem produces explicit Bianchi Eisenstein eigensystems of level $\K_1(\gn)$ and weight $(k,\ell)$ for $k, \ell\in \N$. 

\begin{theorem}{\label{t: dim 1 intro}} Suppose that $S$ has Dirichlet density greater than $1/2$. If $\phi$ has infinity type $[(k+1,0),(-1,\ell)]$ (resp. $[(k+1,\ell+1), (-1,-1)]$ and $(k, \ell)\neq (0,0)$), then $$\mathrm{dim}_{\C}H_{Eis}^1(Y_{\K_1(\gn)},\cV_{k,\ell}^*(\C))[[\gm_{\phi,S}]]=1\;(\mathrm{resp.}\;0),$$
$$\mathrm{dim}_{\C}H_{Eis}^2(Y_{\K_1(\gn)},\cV_{k,\ell}^*(\C))[[\gm_{\phi,S}]]=0\;(\mathrm{resp.}\;1),$$
here $\cV_{k,\ell}^*(\C)$ is the locally constant sheaf attached to the polynomials $V_{k, \ell}^{*}(\C)$ (see \ref{ss: classical cohomology}).
\end{theorem}
\begin{remark}
\begin{enumerate}
    \item Note that $H_{Eis}^0(Y_{\K_f},\cV_{k,\ell}^*(\C))=0$ unless $k=\ell=0$. In this case there is a contribution from those $\phi$ having infinity type $[(1,1), (-1,-1)]$, and this was excluded in the theorem. Moreover, $H_{Eis}^3(Y_{\K_f},\cV_{k,\ell}^*(\C))=0$ for every $k, \ell\in \N$. 
    \item Combining this theorem with \cite{ramakrishnan1994refinement} we obtain results for compact and full singular cohomology but with the stronger hypothesis of Dirichlet density greater than $7/8$.
\end{enumerate}
\end{remark}

\subsection{Results about the Bianchi Eigenvariety}

Henceforth, we fix a prime $p\nmid \fn$ that splits in $K$ as $\gp\overline{\gp}$ and we fix an embedding $\iota_p:\overline{\Q}\hookrightarrow\overline{\Q}_p$ such that $\gp$  corresponds to $\iota_p$. Remark that $\iota_p$ determines a valuation $v_p$ on $\overline{\Q}_p$.

In order to $p$-adically study our Bianchi Eisenstein eigensystems we consider their $p$-stabilizations. More precisely, let $\phi=(\phi_1,\phi_2)$ be a character as in Theorem \ref{t: dim 1 intro}, which gives us a Bianchi Eisenstein eigensystem of level $\K_1(\gn)$ and weight $(k,\ell)$. Denote by  $\alpha_\gq=\phi_1(\gq)^{-1}$ and $\beta_{\gq}=N(\fq)\phi_2(\gq)^{-1}$ the roots of the Hecke polynomial of $\phi$ at $\gq|p$, then a $p$-stabilization of $\phi$ is a couple $\tilde{\phi}= (\phi, (x_{\fq})_{\fq\mid p})$ where $x_{\fq}\in \{\alpha_\gq, \beta_{\gq}\}$ for each $\fq\mid p$. Let $L/\Q_p$ be a finite extension containing the values of $\tilde{\phi}$, we consider the following ideal of $\cH_{\fn, p}\otimes_{\Z}L$ (see def. \ref{d: algebras de hecke}):
$$\gm_{\tilde{\phi}}:=(\{T_\gq-\phi_1(\gq)^{-1}-N(\fq)\phi_2(\gq)^{-1},T_{\gq,\gq}-\phi_1(\gq)^{-1}\phi_2(\gq)^{-1} : \gq\nmid\gn p\} \cup  \{U_\gq-x_\gq:\gq|p\}).
$$

\begin{theorem}\label{t: dim 1 intro p-adic}
Let $\tilde{\phi}$ as above, and suppose $\phi$ has infinity type $[(k+1,0),(-1,\ell)]$ (resp. $[(k+1,\ell+1), (-1,-1)]$ and $(k,\ell)\neq (0, 0)$) then:
$$\mathrm{dim}_{L}H^1(Y_{\K_1(\gn, p)},\cV_{k,\ell}^*(L))_{\fm_{\tilde{\phi}}}=\mathrm{dim}_{L}H_c^2(Y_{\K_1(\gn, p)},\cV_{k,\ell}^*(L))_{\fm_{\tilde{\phi}}}=1 \text{ (resp. $0$)} $$ 
$$\mathrm{dim}_{L}H^2(Y_{\K_1(\gn, p)},\cV_{k,\ell}^*(L))_{\fm_{\tilde{\phi}}}=\mathrm{dim}_{L}H_c^1(Y_{\K_1(\gn, p)},\cV_{k,\ell}^*(L))_{\fm_{\tilde{\phi}}}= 0 \text{ (resp. $1$)}. $$
\end{theorem}
As the cohomology above is concentrated in degrees $1$ and $2$ this theorem gives a complete description of the appearance of the eigenvalues of $\tilde{\phi}$ in the singular and compactly supported cohomology.

Now we take $\tilde{\phi}= (\phi, \alpha_\gp, \beta_{\overline{\gp}})$ when $\phi$ has infinity type $[(k+1,0),(-1,\ell)]$ and $\tilde{\phi}= (\phi, \alpha_\gp, \alpha_{\overline{\gp}})$ when $\phi$ has infinity type $[(k+1,\ell+1), (-1,-1)]$ and $(k,\ell)\neq (0, 0)$. In both cases, these are ordinary $p$-stabilizations (see \ref{r: valuaciones general})

We consider the eigenvarieties $\cE$ (resp $\cE_c$) of tame level $K_1(\gn)$ constructed in \cite{hansen} using the full cohomology (resp. compact supported cohomology) which are endowed with finite maps $\cE\longrightarrow \cW$ (resp. $\cE_c\longrightarrow \cW$) to the $2$-dimensional weight space $\cW$.

\begin{theorem} We have:
\begin{enumerate}
    \item $\tilde{\phi}$ determines a point $x\in \cE$ and the weight map $\cE\longrightarrow \cW$ is etale at $x$.
    \item $\tilde{\phi}$ determines a point $x_c\in \cE_c$ and the weight map $\cE_c\longrightarrow \cW$ is etale at $x_c$.
    \item There exist an open affinoid $\cU\subset \cW$ containing the weight $(k,\ell)$ and a slope $h= (h_{\fp}, h_{\overline{\fp}})\in \R_{\geq 0}^2$ such that if we denote by $\cC \subset \cE_c$ the irreducible component over $\cU$ passing through $x_c$ and $\phi$ has infinity type $[(k+1,0),(-1,\ell)]$ (resp. $[(k+1,\ell+1), (-1,-1)]$ with $(k,\ell)\neq (0, 0)$) then $H^2_c(Y_{\K_1(\gn, p)}, \cD_{\cU})^{\leqslant h}\otimes\cO(\cC)$ (resp. $H^1_c(Y_{\K_1(\gn, p)}, \cD_{\cU})^{\leqslant h}\otimes\cO(\cC)$) is a free $\cO(\cC)$-module of rank $1$. 
\end{enumerate}
\end{theorem}

In (3) we consider $\leqslant h$ part of the cohomology with respect to the Hecke operators at $p$ and tensor product is over the image of the Hecke algebra in the endomorphism of the corresponding cohomology (see \ref{ss: eigenvarieties} for more details). Moreover, the same is true for $x$ but using the full cohomology on degree $1$ when $\phi$ has infinity type $[(k+1,0),(-1,\ell)]$, and the full cohomology on degree $2$ when $\phi$ has infinity type $[(k+1,\ell+1), (-1,-1)]$ with $(k,\ell)\neq (0, 0)$. To prove this theorem we use control theorems on the cohomology of Bianchi manifolds, the multiplicity-one results stated above, and the formalism developed in \cite{BDJ17}. 

An important source of  non-cuspidal Bianchi modular forms is the base change to $\mathrm{GL}_{2, K}$ of cuspidal automorphic representations of $\mathrm{GL}_{2, \Q}$. Our result then implies the etaleness of the ordinary $p$-stabilization. 

\begin{corollary}\label{c: BC intro}
Let $\varphi$ be a Hecke character of $K$ of infinity type $(-k-1,0)$ with $k\geqslant0$ and conductor $\gm$ coprime to $(p)$ satisfying $(\gm,\overline{\gm})=1$. Let $\mathcal{F}$ be the non-cuspidal Bianchi modular form arising from the base change to $K$ of the theta series of $\varphi$, and suppose that its level is exactly $\gm\overline{\gm}$. Then the ordinary $p$-stabilization of $\mathcal{F}$ determines a point $x\in \cE$ (resp $x_c\in \cE_c$) and the weight map $\cE\longrightarrow \cW$ (resp. $\cE_c\longrightarrow \cW$) is etale at $x$ (resp. $x_c$).
\end{corollary}

It is interesting to mention that the other $3$ $p$-stabilizations of $\cF$ are of critical slope, then we cannot apply the results of this work because control theorems on the cohomology are not available in these cases. This makes their study more subtle than that of the ordinary $p$-stabilization (see for example \cite{bellaiche2012critical}, \cite{bergdall2024} and \cite{BW21}). We make some remarks and questions at the end of Section \ref{ss: results eigenvariety} relating them to the Coleman-Mazur eigencurve. We hope to study this in the near future.

Returning to the general case, when $\tilde{\phi}= (\phi, \alpha_\gp, \alpha_{\overline{\gp}})$ and $\phi$ has infinity type $[(k+1,\ell+1), (-1,-1)]$ with $(k,\ell)\neq (0, 0)$, in Section \ref{sss: $p$-adic $L$-functions} we attach $p$-adic $L$-functions and families of them around the point attached to $\tilde\phi$ essentially using the formalism developed in  \cite{chris2017} (see also \cite[\S2.4]{BW21}). It would be interesting to connect these $p$-adic $L$-functions to special values of Hecke $L$-functions. When $p$ is inert in $K$, $\phi=(\phi_1,\phi_2)$ has infinity type $[(k+1,\ell+1), (-1,-1)]$ and the conductors of $\phi_1,\phi_2$ satisfy $p||\gn_1$ and $p\nmid\gn_2$, We emphasize that we obtain an ordinary Bianchi Eisenstein system and all the results of this work can be applied. In particular, after establishing a connection to special values of Hecke $L$-functions of the $p$-adic $L$-function, it would be interesting to check the presence of trivial zeros and to develop the Greenberg-Stevens formalism as in the Hilbert situation.

\subsection*{Acknowledgements} We would like to thank Adel Betina, 
Chris Williams, Mladen Dimitrov, Denis Benois, Andrei Jorza, Héctor del Castillo and Tobias Berger for discussions and comments about this work.

DBS was supported by ECOS230025 and Anid Fondecyt grants 11201025 and 1241702. LP was supported by Proyecto Postdoc Dicyt USA2155\_DICYT, ANR-18-CE40-0029 and Anid Fondecyt Postdoctorado 3240129.

\section{Cohomology of the Bianchi threefold}{\label{s:cohomology of threefolds}}

\subsection{Notations}
Let $p$ be a rational prime and fix throughout the paper embeddings $\iota_\infty:\overline{\Q}\hookrightarrow\C$ and $\iota_p:\overline{\Q}\hookrightarrow\overline{\Q}_p$, note that the latter fixes a $p$-adic valuation $v_p$ on $\overline{\Q}_p$.

Let $K$ be an imaginary quadratic field with discriminant $-D$, ring of integers $\mathcal{O}_K$ and different ideal $\mathcal{D}$ generated by $\delta=\sqrt{-D}$. Denote by $K_{\gq}$ the completion of $K$ with respect to the prime $\gq$ of $K$, $\mathcal{O}_{\gq}$ the ring of integers of $K_{\gq}$ and fix a uniformizer $\varpi_\gq$ at $\gq$.

Throughout the paper, we suppose $p$ splits in $K$ as $\gp\overline{\gp}$ with $\gp$ being the prime corresponding to the embedding $\iota_p$. In particular, if we denote $\cO_{K, p}:= \cO_{K}\otimes_{\Z}\Z_p$ we have $\cO_{K, p}\cong \cO_{\fp}\times \cO_{\overline{\fp}}\cong \Z_p\times \Z_p$.

We denote by $\mathbb{A}_K$ the ring of adeles of $K$, $\mathbb{A}_K^f$ the finite adeles and $\widehat{\cO}_K=\prod_{\gq}\cO_\gq$. If $\chi: K\backslash\A_K^\times\rightarrow \C$ is a Hecke character of conductor $\gf$, we denote by $\chi_{\infty}$, $\chi_f$ and $\chi_{\gq}$ the restriction of $\chi$ to $\C$, $\mathbb{A}_K^f$ and $K_{\gq}^{\times}$ respectively. For an ideal $\gm\subset\cO_K$ write $\chi_\gm:=\prod_{\gq|\gm}\chi_\gq$.

A Hecke character $\chi$ as above can be identified with a function on ideals of $K$ that has support on those that are coprime to $\gf$, concretely, for each prime ideal $\gq\nmid\gf$, we define $\chi(\gq)=\chi_\gq(\varpi_\gq)$ (which is independent of the choice of the uniformizer) and $\chi(\gq)=0$ otherwise. We define $\chi^c:=\chi\circ c$ where $c$ denotes complex conjugation on the ideals
of $K$. Denote by $\mathrm{Cl}_K(\gf):=K^\times\backslash\A_K^\times/I(\gf)\C^\times$ the ray class group of $K$ modulo $\gf$, where $I(\ff):=\{x\in\widehat{\cO}_K^\times:x\equiv1\text{ (mod $\ff$)}\}$.

Let $B$ be the Borel subgroup of upper triangular matrices of $GL_{2/K}$, with unipotent radical $U$ and maximal torus $T$. For an ideal $\gn\subset\cO_K$, denote by $\K_1(\gn)$ (resp. $\K_0(\fn)$) the subgroup of $\mathrm{GL}_2(\widehat{\cO}_K)$ of matrices congruent to $\smallmatrixx{*}{*}{0}{1}$ (resp. $\smallmatrixx{*}{*}{0}{*}$) modulo $\gn$. Let $\gr$ be an ideal coprime to $\gn$, we will also consider the group $\K_1(\gn,\gr):=\K_1(\gn)\cap \K_0(\gr)$. Moreover, for every prime ideal $\fq$ we write $\K_1(\gn)_\gq\subset \mathrm{GL}_2(\cO_\gq)$ (resp. $\K_0(\fn)_\gq\subset \mathrm{GL}_2(\cO_\gq)$) for the subgroup of matrices congruent to $\smallmatrixx{*}{*}{0}{1}$ (resp. $\smallmatrixx{*}{*}{0}{*}$) modulo $\fq^{s_\gq}$ where $\gq^{s_\gq}||\gn$. Note that for a prime $\gr$, the group $\K_0(\gr)_\gr=I_\gr$ is just the Iwahori subgroup at $\gr$.

\begin{definition} \label{d: algebras de hecke}
For $\gn\subset\cO_K$ an ideal (resp. coprime to $p$) consider the abstract Hecke algebra: 
$$\cH_\gn= \Z[\{T_\gq,T_{\gq,\gq} : (\gq,\gn)=1\}],$$
$$ (\text{resp. }\cH_{\gn,p}= \Z[\{T_\gq,T_{\gq,\gq} : (\gq,p\gn)=1\}\cup \{U_\gq :\gq|p\} ]).$$
\end{definition}

Let $N$ be a module over a ring $R$ and endowed with a linear action by $\cH_\gn$. By a \textit{system of Hecke eigenvalues} appearing in $N$ we mean a ring homomorphism $\psi:\cH_\gn\rightarrow R$ such that there exists a non-zero vector $n\in N$ satisfying $t\cdot h=\psi(t)\cdot h$ for all $t\in\cH_\gn$.

\subsection{Induced representations}{\label{ss: induced representations}}

\subsubsection{Definitions}
Let $\chi_1, \chi_2: K^\times\backslash\A_K^\times\rightarrow\C^\times$ be two Hecke characters of conductor $\gm_1$, $\gm_2$ and infinity type $(k_1,\ell_1)$, $(k_2,\ell_2)$ respectively. They determine a character $\chi=(\chi_1,\chi_2)$ in $T(K)\backslash T(\A_K)$ of infinity type $[(k_1,\ell_1),(k_2,\ell_2)]$ and finite part $\chi_f=(\chi_{1,f},\chi_{2,f})$.

For each prime $\gq$, $(\chi_1,\chi_2)_\gq=(\chi_{1,\gq},\chi_{2,\gq})$ defines a character
\begin{align}\label{rep borel}
\chi_\gq:B(K_\gq)&\rightarrow\C^\times \\
\nonumber \smallmatrixx{a}{b}{0}{d}&\mapsto \chi_{1,\gq}(a)\chi_{2,\gq}(d),
\end{align}
and we denote by 
\begin{equation*}
V_{\chi_\gq}=\{h:\mathrm{GL_2}(K_\gq)\rightarrow\C: h(b g)=\chi_\gq(b)h(g), \forall b\in B(K_\gq), g\in \mathrm{GL_2}(K_\gq)\},   
\end{equation*}
the corresponding local induced representation.

We consider the right action of $\mathrm{GL_2}(K_\gq)$ on $V_{\chi_\gq}$ given by $(h\cdot g)(x)= h(x\cdot g^{-1})$ for $g, x\in \mathrm{GL_2}(K_\gq)$. For example, if $\K_\gq$ is a compact subgroup of $\mathrm{GL}_2(K_\gq)$ then we have $V_{\chi_\gq}^{\K_\gq}=\{h\in V_{\chi_\gq}: h(gk)=h(g), \forall k\in \K_\gq, g \in \mathrm{GL_2}(K_\gq) \}$. 

\begin{remark} \label{r: no cero entonces conductor 1} Note that if $V_{\chi_\gq}^{\mathrm{GL}_2(\cO_{\fq})}\neq \{0\}$ then $\chi_{1, \fq}$ and $\chi_{2, \fq}$ are necessarily unramified characters.
\end{remark}

Similarly, we define the space $V_{\chi_f}$ which is endowed with a right action of $\mathrm{GL}_2(\A_K^f)$ in the same way as before. Note that if $\K_f=\prod_{\gq\nmid \infty}\K_\gq\subset \mathrm{GL}_2(\A_K^f)$ is an open compact subgroup, then we have
$V_{\chi_f}^{\K_f}=\bigotimes_{\gq\nmid \infty} V_{\chi_\gq}^{\K_\gq}$.

For each prime ideal $\gq\subset \cO_K$ we define the Hecke operators $T_{\fq}$, $T_{\fq,\fq}$ as follows. Let $x=\smallmatrixx{1}{0}{0}{\varpi_\gq}\in \mathrm{GL}_2(\A_K^{f})$, where we consider $\varpi_\gq\in\A_K^{f,\times}$ as being $\varpi_\gq$ at $\fq$ and trivial outside $\fq$, then for $h \in V_{\chi_f}^{\K_f}$ and each $g\in \mathrm{GL}_2(\A_K^{f})$ we put: 
\begin{equation*}
(T_\gq\cdot h)(g)= \sum_{\gamma\in\K_fx\K_f/\K_f}h(g\gamma^{-1}),
\end{equation*}
\begin{equation*}
(T_{\gq,\gq}\cdot h)(g)= h(g \smallmatrixx{\varpi_\gq}{0}{0}{\varpi_\gq}^{-1}).
\end{equation*}
Observe that these operators are independent of the choice of uniformizer. When $\K_f =\K_1(\gn)$ we write $U_\gq$ instead $T_\gq$ for primes $\gq|\gn$. In particular, we obtain an action of the Hecke algebra $\cH_{\fn}\otimes_{\Z}\C$. The same observations follow when $\K_f =\K_1(\gn,\gr)$ for $\fr$ some ideal and $\fq\nmid \fn\gr$. In this case, we obtain an action of the Hecke algebra $\cH_{\fn, p}\otimes_{\Z}\C$.

\begin{remark}
Note that the action of the Hecke operators $T_\gq$ above is consistent with \cite[\S 2.9.3]{berger}. Moreover, it differs from the action in \cite[\S 3.3]{berger2008denominators}, which is motivated by our interest in the right-$\mathrm{GL}_2(K)$-module $V_{k,\ell}^*(\C)$ instead of the left-$\mathrm{GL}_2(K)$-module $V_{k,\ell}(\C)$ (see Section \ref{ss: classical cohomology} and Remark \ref{moduloberger}). 
\end{remark}

\subsubsection{Multiplicity-one}\label{sss:  mult one}

\begin{lemma}\label{l: eigen} Let $\chi_1,\chi_2$ be Hecke characters of $K$ of conductors $\gm_1$, $\gm_2$ respectively such that $(\gm_1,\gm_2)=1$ and set $\gm=\gm_1\gm_2$. Then $V_{\chi_f}^{\K_1(\gm)}$ is a 1-dimensional $\C$-vector space and each $h\in V_{\chi_f}^{\K_1(\gm)}$ is an eigenvector of:
\begin{itemize}
\item $T_\gq$ for $\gq\nmid\gm$ (resp. $U_\gq$ for $\gq|\gm$), with eigenvalues $N(\fq) \chi_2(\gq)^{-1}+  \chi_1(\gq)^{-1};$ 
\item $T_{\gq,\gq}$ for $\gq\nmid\gm$, with eigenvalues $\chi_1(\gq)^{-1}\chi_2(\gq)^{-1};$
\end{itemize}
\end{lemma}

\begin{proof} The proof of this lemma follows exactly the arguments of \cite[\S 3.2]{berger2008denominators} and the calculations given in \cite[Lemma 3.11]{berger}. Since the level $\K^1(\gn)$ used in Berger's works differs from the level $\K_1(\gn)$ we use, then we prefer to describe his work in our setting.  

First note that $V_{\chi_f}^{\K_1(\gm)}=\bigotimes_{\gq\nmid \infty} V_{\chi_\gq}^{\K_1(\gm)_\gq}$, and by \cite[Thm 1]{casselman1973some} we know that $V_{\chi_\gq}^{\K_1(\fq^{s_{\fq}})}$ is a one dimensional $\C$-vector space. 

For the study of the Hecke eigenvalues we produce an explicit non-zero vector in $V_{\chi_f}^{\K_1(\gm)}$, by considering the different cases of $\fq$.

If $\gq\nmid \gm$ then we define $h_{\fq}\in V_{\chi_\gq}$ given by the formula: 
$$h_\fq(g)= \chi_{1, \fq}(a)\chi_{2, \fq}(d)$$
where $g= \smallmatrixx{a}{b}{0}{d}\cdot k$ with $k\in \K_1(\gm)_\gq=\mathrm{GL}_2(\cO_{\fq})$. Note that $h_{\fq}\in V_{\chi_\gq}^{\mathrm{GL}_2(\cO_{\fq})}$.

Now we consider the case $\gq\mid \gm$, for which we use the following decomposition: 
$$\mathrm{GL}_2(K_\fq)=  \bigsqcup_{i= 0}^{s_\gq} B(K_{\fq})\smallmatrixx{1}{0}{\varpi_{\fq}^{i}}{1} \K_1(\gm)_\gq,$$
where $\gq^{s_\gq}||\gm$.

If $\gq\mid \gm_1$ then we define $h_{\fq}\in V_{\chi_\gq}$ given by the formula:
$$
h_\fq(g)= \begin{cases} 
			\chi_{1, \fq}(a) \chi_{2, \fq}(d) & \text{if $g= \smallmatrixx{a}{b}{0}{d}\cdot k\in B(K_{\fq}) \K_1(\gm)_\gq$}\\
           0  & \text{if not,}
\end{cases}
$$
and note that $h_{\fq}\in V_{\chi_\gq}^{\K_1(\gm)_\gq}$.

If $\gq\mid \gm_2$ then we define $h_{\fq}\in V_{\chi_\gq}$ given by the formula:
$$
h_\fq(g)= \begin{cases}
			\chi_{1, \fq}(a) \chi_{2, \fq}(d) & \text{if $g= \smallmatrixx{a}{b}{0}{d}\cdot \smallmatrixx{1}{0}{1}{1}\cdot k\in B(K_{\fq})\smallmatrixx{1}{0}{1}{1}  \K_1(\gm)_\gq$}\\
           0  & \text{if not,}
\end{cases}
$$
and note that $h_{\fq}\in V_{\chi_\gq}^{\K_1(\gm)_\gq}$.

By construction, we have $h:= \prod_{\fq\nmid \infty} h_{\fq}\in V_{\chi_f}^{\K_1(\gm)}$. To obtain the Hecke-eigenvalues we calculate the action of the Hecke operators on $h$. 

If $\gq\nmid \gm$ we have that $(T_\fq\cdot h)(1)=$
\begin{equation}\label{e: operador hecke dimension 1}
=\sum_{a\in \cO_\fq/ \varpi_\fq} h\smallmatrixx{1}{-a\varpi_{\fq}^{-1}}{0}{\varpi_{\fq}^{-1}}+ h\smallmatrixx{\varpi_{\fq}^{-1}}{0}{0}{1}= \left(\sum_{a\in \cO_\fq/ \varpi_\fq} h_{\fq}\smallmatrixx{1}{-a\varpi_{\fq}^{-1}}{0}{\varpi_{\fq}^{-1}}+ h_{\fq}\smallmatrixx{\varpi_{\fq}^{-1}}{0}{0}{1}\right)\cdot \prod_{\fl\neq \fq}h_{\fl}(1) 
\end{equation}
$$= \left(N(\fq)\chi_{2, \fq}(\varpi_\fq)^{-1}+ \chi_{1, \fq}(\varpi_\fq)^{-1} \right) \cdot \prod_{\fl\neq \fq}h_{\fl}(1)= \left(N(\fq)\chi_{2, \fq}(\varpi_\fq)^{-1}+ \chi_{1, \fq}(\varpi_\fq)^{-1} \right)\cdot h(1),$$
then, the $T_\gq$-eigenvalue is $N(\fq)\chi_{2, \fq}(\varpi_\fq)^{-1}+ \chi_{1, \fq}(\varpi_\fq)^{-1}$.

If $\gq\mid \gm_1$ we have: 
$$(U_\fq\cdot h)(1)= \sum_{a\in \cO_\fq/ \varpi_\fq} h\smallmatrixx{1}{-a\varpi_{\fq}^{-1}}{0}{\varpi_{\fq}^{-1}}= \left(\sum_{a\in \cO_\fq/ \varpi_\fq} h_{\fq}\smallmatrixx{1}{-a\varpi_{\fq}^{-1}}{0}{\varpi_{\fq}^{-1}}\right)\cdot \prod_{\fl\neq \fq}h_{\fl}(1)$$
$$= \left(N(\fq)\chi_{2, \fq}(\varpi_\fq)^{-1} \right) \cdot \prod_{\fl\neq \fq}h_{\fl}(1)= N(\fq)\chi_{2, \fq}(\varpi_\fq)^{-1}\cdot h(1)$$
then, the $U_\gq$-eigenvalue is $N(\fq)\chi_{2, \fq}(\varpi_\fq)^{-1}$.

If $\gq\mid \gm_2$, we have:
$$(U_\gq h)\smallmatrixx{0}{1}{1}{0}=\left(\sum_{a\in \cO_\fq/ \varpi_\fq} h_{\fq}\smallmatrixx{0}{1}{1}{0}\smallmatrixx{1}{-a\varpi_{\fq}^{-1}}{0}{\varpi_{\fq}^{-1}}\right)\cdot \prod_{\fl\neq \fq}h_{\fl}\smallmatrixx{0}{1}{1}{0}$$
$$=\left(h_\gq\smallmatrixx{\varpi_\gq^{-1}}{0}{0}{1}\smallmatrixx{0}{1}{1}{0}+ \sum_{a\in \cO_\fq/ \varpi_\fq, a\neq 0} h_{\fq}\smallmatrixx{a^{-1}}{-\varpi_{\fq}^{-1}}{0}{a\varpi_{\fq}^{-1}}\smallmatrixx{1}{0}{\varpi_\gq a^{-1}}{-1}\right)\cdot\prod_{\fl\neq \fq}h_{\fl}\smallmatrixx{0}{1}{1}{0},$$
and since $h_\gq\smallmatrixx{\varpi_\gq^{-1}}{0}{0}{1}\smallmatrixx{0}{1}{1}{0}=h_\gq\smallmatrixx{-\varpi_\gq^{-1}}{(1+\varpi_\gq)\varpi_\gq^{-1}}{0}{1}\smallmatrixx{1}{0}{1}{1}\smallmatrixx{1+\varpi_\gq}{-1}{\varpi_\gq}{1}=\chi_{1, \fq}(-\varpi_\gq^{-1})$ and the sum over $a\in \cO_\fq/ \varpi_\fq$ non-trivial  is zero, then 
$$(U_\gq h)\smallmatrixx{0}{1}{1}{0}=\chi_{1, \fq}(-\varpi_\gq^{-1}) \cdot\prod_{\fl\neq \fq}h_{\fl}\smallmatrixx{0}{1}{1}{0}.$$
Now, since
$$h_\gq\smallmatrixx{0}{1}{1}{0}=h_\gq \smallmatrixx{-1}{1+\varpi_\gq}{0}{1}\smallmatrixx{1}{0}{1}{1}\smallmatrixx{1+\varpi_\gq}{-1}{-\varpi_\gq}{1}=\chi_{1, \fq}(-1)$$
we have 
$$(U_\gq h)\smallmatrixx{0}{1}{1}{0}=\chi_{1, \fq}(\varpi_\gq)^{-1} h\smallmatrixx{0}{1}{1}{0},$$
and the $U_\gq$-eigenvalue is $\chi_{1, \fq}(\varpi_\fq)^{-1}$.

Finally, for $\gq\nmid\gm$ we have 
$$(T_{\gq,\gq}\cdot h)(1)= 
h\smallmatrixx{{\varpi_\gq}^{-1}}{0}{0}{{\varpi_\gq}^{-1}}=\chi_1(\varpi_\gq)^{-1}\chi_2(\varpi_\gq)^{-1}\cdot h(1),$$
then, the $T_{\gq,\gq}$-eigenvalue is $\chi_1(\gq)^{-1}\chi_2(\gq)^{-1}$.
\end{proof}

Let $\phi=(\phi_1,\phi_2)$ be a couple of Hecke characters of $K$ of conductors $\fn_1$, $\fn_2$ respectively, such that $(\gn_1,\gn_2)=1$. Let $S$ be a set of primes of $K$ not dividing $\gn:=\fn_1 \fn_2$. Motivated by the previous lemma we consider the following ideal of $\cH_{\fn}\otimes_{\Z} \C$:
\begin{equation} \label{e: ideal}
 \gm_{\phi, S}= (\{T_\gq-\phi_1(\gq)^{-1}-N(\fq)\phi_2(\gq)^{-1},T_{\gq,\gq}-\phi_1(\gq)^{-1}\phi_2(\gq)^{-1} : \gq\in S\}).
\end{equation}

Moreover, we will just write $\gm_{\phi}$ when $S=\{\gq \text{ prime }: \gq\nmid\gn\}$. 

\begin{theorem}\label{t: dim1borde} Suppose that $S$ has Dirichlet density greater than $1/2$ and $\phi$ as above with infinity type $[(k_1,\ell_1),(k_2,\ell_2)]$ satisfying $(k_1,\ell_1)\neq (k_2+1,\ell_2+1)$. Let $\chi_1,\chi_2$ be a pair of Hecke characters such that $V_{(\chi_1,\chi_2)_f}^{\K_1(\fn)}[[\gm_{\phi, S}]]\neq \{0\}$, then either $(\chi_1,\chi_2)=(\phi_1,\phi_2)$ or $(\chi_1,\chi_2)= (\phi_2|\cdot|_{\A_K},\phi_1|\cdot|_{\A_K}^{-1})$. Moreover, we have that $V_{(\chi_1,\chi_2)_f}^{\K_1(\fn)}[[\gm_{\phi, S}]]$ is a 1-dimensional $\C$-vector space. 
\end{theorem}

\begin{proof} 
Denote by $\gf_1,\gf_2$ the conductor of $\chi_1,\chi_2$ respectively. Shrinking $S$ we can suppose that every prime $\fq\in S$ has degree $1$ and does not divide $\gf_1\gf_2\cD$. Note that $S$ still has the same Dirichlet density since the set of primes of degree $1$ have Dirichlet density 1 (see for example \cite{lang1994algebraic}[Ch. VIII \S 4]).

Suppose $V_{(\chi_1,\chi_2)_f}^{\K_1(\fn)}[[\gm_{\phi, S}]]\neq \{0\}$ and let $h\in V_{(\chi_1,\chi_2)_f}^{\K_1(\fn)}[[\gm_{\phi, S}]]$ be a non-trivial vector. Then for each $\gq\in S$ there exists $c\in\Z_{>0}$, such that $(T_\gq-\phi_1(\gq)^{-1}-N(\fq)\phi_2(\gq)^{-1})^c h=0$ and $(T_{\gq,\gq}-\phi_1(\gq)^{-1}\phi_2(\gq)^{-1})^c h=0$.

By the proof of the Lemma \ref{l: eigen} we also have that $T_\gq\cdot h = (\chi_1(\gq)^{-1}+N(\gq)\chi_2(\gq)^{-1})h$ and $T_{\gq,\gq}\cdot h = \chi_1(\gq)^{-1}\chi_2(\gq)^{-1}h$ for all $\gq\nmid\gn\gf_1\gf_2$. From these equations, in particular, we have for each prime $\gq\in S$:
\begin{itemize}
\item  $(\chi_1(\gq)^{-1}+N(\gq)\chi_2(\gq)^{-1}-\phi_1(\gq)^{-1}-N(\gq)\phi_2(\gq)^{-1})^c\cdot h=0$;
\item $(\chi_1(\gq)^{-1}\chi_2(\gq)^{-1}-\phi_1(\gq)^{-1}\phi_2(\gq)^{-1})^c\cdot h=0$.
\end{itemize}

Since $h\neq 0$, we deduce that 
\begin{equation}\label{e: tq}
\chi_1(\gq)^{-1}+N(\gq)\chi_2(\gq)^{-1}=\phi_1(\gq)^{-1}+N(\gq)\phi_2(\gq)^{-1},   \text{ for all $\gq\in S$;}
\end{equation}
\begin{equation}\label{e: diam}
\chi_1(\gq)^{-1}\chi_2(\gq)^{-1}=\phi_1(\gq)^{-1}\phi_2(\gq)^{-1}, \text{ for all $\gq\in S$}. 
\end{equation}

Multiplying (\ref{e: tq}) by $\chi_1(\gq)^{-1}$, substituting (\ref{e: diam}) and factorizing, for each $\gq\in S$ we have
\begin{equation}\label{e: final}
(\chi_1(\gq)^{-1}-\phi_1(\gq)^{-1})(\chi_1(\gq)^{-1}-N(\gq)\phi_2(\gq)^{-1})=0. 
\end{equation}

Suppose both factors on equation (\ref{e: final}) vanish for some prime $\fr\in S$, then the characters $(\chi_1^{-1}\phi_1)_\fr,(\chi_1^{-1}\phi_2|\cdot|_{\A_K})_{\fr}:K_{\fr}^\times\rightarrow\C^\times$ are trivial, and we obtain from \cite[Prop 3]{rajan2000refinement} that the characters
$\chi_1^{-1}\phi_1$ and $\chi_1^{-1}\phi_2|\cdot|_{\A_K}$ have finite order. This implies that $\chi_1$ has infinity type $(k_1,\ell_1)$ and $(k_2+1,\ell_2+1)$ at the same time, which contradicts the hypothesis on the infinity types of $\phi_1$ and $\phi_2$. Note that once either $\chi_1(\fr)^{-1}-\phi_1(\fr)^{-1}=0$ or $\chi_1(\fr)^{-1}-N(\fr)\phi_2(\fr)^{-1}=0$, this determines the infinity type of $\chi_1$ and then either $\chi_1(\gq)^{-1}-\phi_1(\gq)^{-1}=0$ or $\chi_1(\gq)^{-1}-N(\gq)\phi_2(\gq)^{-1}=0$ for all $\gq\in S$

Suppose that the first factor of (\ref{e: final}) vanishes, i.e,
$\chi_1(\gq)=\phi_1(\gq)$ for all $\gq\in S$. Since $\chi_{1,\infty}=\phi_{1,\infty}$, we have that $\chi_1\phi_1^{-1}$ defines a character on $\mathrm{Cl}_K(\gf)$ where $\gf=\gn\gf_1\gf_2$, then to prove the Hecke characters $\chi_1,\phi_1$ are equal it suffices to prove that $\chi_1\phi_1^{-1}$ is the trivial character on $\mathrm{Cl}_K(\gf)$. 

Let $\gq_1,\gq_2,..\gq_t$ be prime ideals representing $\mathrm{Cl}_K(\gf)$. Similarly as the proof of \cite[Thm 3.6]{atwill2013newform}, if $\chi_1(\gq_i)=\phi_1(\gq_i)$ for $s$ values of $i$, then the density of primes for which $\chi_1(\gq)=\phi_1(\gq)$ is at most $s/t$. It now follows that $\chi_1(\gq_i)=\phi_1(\gq_i)$ for
more than $t/2$ values of $i$.

The orthogonality relations on characters show that $\sum_{i=1}^t\chi_1\phi_1^{-1}(\gq_i)=0$ unless $\chi_1\phi_1^{-1}$ is the trivial character on $\mathrm{Cl}_K(\gf)$. Since we showed that $\chi_1\phi_1^{-1}(\gq_i)=1$ for more than $t/2$ of the $\gq_i$, then $\chi_1\phi_1^{-1}$ is the trivial character on $\mathrm{Cl}_K(\gf)$ and we conclude that $\chi_1=\phi_1$. 

By equation (\ref{e: diam}), we have that $\chi_2^{-1}(\gq)\phi_2(\gq)=1$ for all $\gq\in S$, therefore $(\chi_2^{-1}\phi_2)_\gq$ is trivial and as before, using \cite[Prop 3]{rajan2000refinement} we obtain that $\chi_2^{-1}\phi_2$ has finite order. Thus $\chi_2^{-1}\phi_2$ determines a character on $\mathrm{Cl}_K(\gf)$, which by equation (\ref{e: diam}), is the trivial character, and we deduce that $\chi_2=\phi_2$.

If the second factor of equation (\ref{e: final}) vanishes for all $\gq\in S$, we proceed exactly as before to deduce that $\chi_1=\phi_2|\cdot|_{\A_K}$ and $\chi_2=\phi_1|\cdot|_{\A_K}^{-1}$.

We conclude that for each case $(\chi_1,\chi_2)=(\phi_1,\phi_2)$ or $(\chi_1,\chi_2)= (\phi_2|\cdot|_{\A_K},\phi_1|\cdot|_{\A_K}^{-1})$, the space $V_{(\chi_1,\chi_2)_f}^{\K_1(\fn)}[[\gm_{\phi, S}]]$ is a 1-dimensional $\C$-vector space by Lemma \ref{l: eigen}.
\end{proof}

\begin{remark}\label{r: 0dim result}
\begin{enumerate}
    \item[i)] We will apply this theorem to obtain a multiplicity-one result on the cohomology of our threefold in Theorem \ref{t: main theorem}. Observe that the hypothesis on the weights above are satisfied (see Remark \ref{r: dimesnion 1 borde}).
    \item[ii)] From the above theorem we deduce directly that if the infinity type of $(\chi_1,\chi_2)$ is different to $[(k_1,\ell_1),(k_2,\ell_2)]$ and $[(k_2+1,\ell_2+1),(k_1-1,\ell_1-1)]$, then $V_{(\chi_1,\chi_2)_f}^{\K_1(\fn)}[[\gm_{\phi, S}]]=\{0\}$, which we record here for future reference.
\end{enumerate}
\end{remark}

\subsubsection{Multiplicity-one with level at $p$}\label{sss: mult one old}

The goal of this section is to prove multiplicity-one results as in the previous section for eigensystems which are old at $p$. We start with the following corollary of Lemma \ref{l: eigen}.
\begin{corollary}\label{c: p-estab}
Let $\chi_1,\chi_2$ be as in Lemma \ref{l: eigen} and let $\fr$ be a prime not dividing $\gm$. 
Then there is a Hecke equivariant decomposition $V_{\chi_f}^{\K_1(\gm,\fr)}=V_1\oplus V_2$ with $V_1,V_2$ 1-dimensional $\C$-vector spaces. Moreover, if  $i\in\{1,2\}$ then any $h\in V_i$ is an eigenvector of:
\begin{itemize}
\item $T_\gq$ for $\gq\nmid\gm\fr$ (resp. $U_\gq$ for $\gq|\gm$) with eigenvalues $N(\fq) \chi_2(\gq)^{-1}+  \chi_1(\gq)^{-1};$
\item $U_\fr$ with eigenvalue 
$N(\fr)^{i-1}\chi_i(\fr)^{-1};$ 
\item $T_{\gq,\gq}$ for $\gq\nmid\gm\fr$, with eigenvalues $\chi_1(\gq)^{-1}\chi_2(\gq)^{-1}.$ 
\end{itemize}
\end{corollary}
\begin{proof}
First note that for each prime $\gq\neq\gr$ the local  representations $V_{\chi_\gq}^{\K_1(\gm)_\gq}$ and $V_{\chi_\gq}^{\K_1(\gm,\fr)_\gq}$ are equal, then we are only interested in the local representation $V_{\chi_\gr}^{\K_1(\gm,\fr)_\gr}=V_{\chi_\gr}^{I_\gr}$.

By Bruhat decomposition for $\mathrm{GL}_2(K_\gr)$, $I_\gr$ has two orbits on $\mathbb{P}^1(\cO_\gr)$ represented by $id=\smallmatrixx{1}{0}{0}{1}$ and $w=\smallmatrixx{0}{1}{1}{0}$. In particular,
$$B(K_\gr)\backslash \mathrm{GL}_2(K_\gr)/ I_\gr= \{[id],[w]\},$$ and then
$V_{\chi_\gr}^{I_\gr}$ is a 2-dimensional $\C$-vector space. Now, we describe an explicit basis of $U_{\fr}$-eigenvectors of $V_{\chi_\gr}^{I_\gr}$ by defining  $h_1,h_2\in V_{\chi_\gr}$ by the formula:
$$
h_1(g)= \begin{cases}
			\chi_{1, \fq}(a) \chi_{2, \fq}(d) & \text{if $g= \smallmatrixx{a}{b}{0}{d}\cdot k\in B(K_{\fq}) I_\gr$}\\
           0  & \text{if not,}
\end{cases}
$$
$$
h_2(g)= \begin{cases}
			\chi_{1, \fq}(a) \chi_{2, \fq}(d) & \text{if $g= \smallmatrixx{a}{b}{0}{d}\cdot w\cdot k\in B(K_{\fq})w I_\gr$}\\
           0  & \text{if not,}
\end{cases}
$$
noting that $h_1,h_2\in V_{\chi_\gr}^{I_\gr}$ and by the proof of Lemma \ref{l: eigen} each of them has the $U_\gr$-eigenvalue desired.
\end{proof}

Let $\phi= (\phi_1, \phi_2)$ be a couple of Hecke characters of $K$, with $\phi_1,\phi_2$ of conductors $\fn_1,\fn_2$ respectively, such that $(\gn_1,\gn_2)=1$ and $p\nmid\gn:= \fn_1 \fn_2$. For $\gq|p$ we consider the Hecke polynomial of $\phi$ given by 
\begin{equation}\label{e: hecke polynomial}
    x^2-(\phi_1(\gq)^{-1}+\phi_2(\gq)^{-1}N(\gq))x+ \phi_1(\gq)^{-1}\phi_2(\gq)^{-1}N(\gq)
\end{equation}
which has roots $\alpha_\gq=\phi_2(\gq)^{-1}N(\gq)$ and $\beta_\gq=\phi_1(\gq)^{-1}$.  

Motivated by the previous corollary we introduce the following ideal. For any set $S$ of primes of $K$ not dividing $p\gn$, and choices of $x_{\fq}\in \{\alpha_{\fq}, \beta_{\fq}\}$ for each $\fq\mid p$, we consider the following ideal of $\cH_{\fn,p}\otimes_{\Z}\C$: 
\begin{align} \label{e: ideal p-estabilizado}
    \gm_{\tilde{\phi},S}:= & (\{T_\gq-\phi_1(\gq)^{-1}-N(\fq)\phi_2(\gq)^{-1},T_{\gq,\gq}-\phi_1(\gq)^{-1}\phi_2(\gq)^{-1} : \gq\in S\} \\ 
\nonumber & \cup  \{U_\gq-x_\gq:\gq|p\}).
\end{align}

Moreover, we will just write $\gm_{\tilde\phi}$ when $S=\{\gq \text{ prime }: \gq\nmid p\gn\}$. 

\begin{corollary}\label{c: dim1borde} 
Suppose $S$ has Dirichlet density greater than $1/2$ and let $\phi_1,\phi_2,\chi_1,\chi_2$ be Hecke characters as in Theorem \ref{t: dim1borde} such that $p\nmid\gn$. If $V_{(\chi_1,\chi_2)_f}^{\K_1(\fn,p)}[[\gm_{\tilde{\phi},S}]]\neq \{0\}$ then either $(\chi_1,\chi_2)=(\phi_1,\phi_2)$ or $(\chi_1,\chi_2)= (\phi_2|\cdot|_{\A_K},\phi_1|\cdot|_{\A_K}^{-1})$. Moreover, we have that $V_{(\chi_1,\chi_2)_f}^{\K_1(\fn,p)}[[\gm_{\tilde{\phi},S}]]$ is a 1-dimensional $\C$-vector space. 
\end{corollary}
\begin{proof}
First, note that the same proof of Theorem \ref{t: dim1borde} gives us that either $(\chi_1,\chi_2)=(\phi_1,\phi_2)$ or $(\chi_1,\chi_2)= (\phi_2|\cdot|_{\A_K},\phi_1|\cdot|_{\A_K}^{-1})$.

By using Corollary \ref{c: p-estab} twice, we have $$V_{(\phi_1,\phi_2)_f}^{\K_1(\gn,p)}\cong V^{\alpha_\gp,\alpha_{\overline{\gp}}}\oplus V^{\alpha_\gp,\beta_{\overline{\gp}}}\oplus V^{\beta_\gp,\alpha_{\overline{\gp}}}\oplus V^{\beta_\gp,\beta_{\overline{\gp}}}, $$ with $V^{x_\gp,x_{\overline{\gp}}}$ a 1-dimensional $\C$-vector space for $x_{\gq}\in \{\alpha_{\gq}, \beta_{\gq}\}$ with $\gq|p$.  Moreover, let $h\in V^{x_\gp,x_{\overline{\gp}}}$, then 
\begin{itemize}
\item $T_\gq\cdot h=(\phi_1(\gq)^{-1}+N(\fq)\phi_2(\fq)^{-1})h$ for all $\gq\nmid\gn p$,
\item $T_{\gq,\gq}\cdot h=\phi_1(\gq)^{-1}\phi_2(\gq)^{-1}h$ for all $\gq\nmid\gn p$,
\item $U_\gq\cdot h=x_\gq h$ for $\gq|p$.
\end{itemize}
Since the spaces $V_{(\phi_1,\phi_2)_f}^{\K_1(\gn,p)}$ and $V_{(\phi_2|\cdot|_{\A_K},\phi_1|\cdot|_{\A_K}^{-1})_f}^{\K_1(\fn,p)}$ have the same eigenvalues, the latter can be written as a direct sum of four 1-dimensional $\C$-vector spaces having respectively the same eigenvalues of the $V^{x_\gp,x_{\overline{\gp}}}$'s above.

We conclude that for each case $(\chi_1,\chi_2)=(\phi_1,\phi_2)$ or $(\chi_1,\chi_2)= (\phi_2|\cdot|_{\A_K},\phi_1|\cdot|_{\A_K}^{-1})$, the space
$V_{(\chi_1,\chi_2)_f}^{\K_1(\fn,p)}[[\gm_{\tilde{\phi},S}]]\cong V^{x_\gp,x_{\overline{\gp}}}$ is a 1-dimensional $\C$-vector space.  
\end{proof}

\begin{remark}\label{r: 0dim result p-estab}
Analogously as part ii) of Remark \ref{r: 0dim result}, we record here for future reference that from the above corollary we deduce that if the infinity type of $(\chi_1,\chi_2)$ is different to $[(k_1,\ell_1),(k_2,\ell_2)]$ and $[(k_2+1,\ell_2+1),(k_1-1,\ell_1-1)]$, then $V_{(\chi_1,\chi_2)_f}^{\K_1(\fn)}[[\gm_{\tilde{\phi},S}]]=\{0\}$.
\end{remark}

\subsection{Cohomology of threefolds}\label{ss: classical cohomology}

Let $\K_f$ be an open compact subgroup of $\mathrm{GL}_2(\widehat{\cO}_K)$ and put $\K_\infty=\mathrm{SU}_2(\C)\C^\times$, we define the Bianchi threefold of level $\K_f$ to be the locally symmetric space
\begin{equation*}
Y_{\K_f}:=\mathrm{GL}_2(K)\backslash\mathrm{GL}_2(\A_K)/\K_\infty\K_f.
\end{equation*}

Let $\cR$ be a sheaf on $Y_{\K_f}$, we denote by $H^i(Y_{\K_f},\cR)$  (resp. $H_c^i(Y_{\K_f},\cR)$) the $i$-th singular cohomology group of $\cR$ (resp. with compact support). There is a homotopy-equivalence $Y_{\K_f}\hookrightarrow\overline{Y}_{\K_f}$, where $\overline{Y}_{\K_f}$ denotes the Borel-Serre compactification of $Y_{\K_f}$, and we obtain the long exact sequence
\begin{equation}\label{e: exact sequence}
...\rightarrow H_c^i(Y_{\K_f},\cR)\rightarrow H^i(Y_{\K_f},\cR)\rightarrow H^i(\partial \overline{Y}_{\K_f},\tilde{\cR})\rightarrow H_c^{i+1}(Y_{\K_f},\cR)\rightarrow...
\end{equation}
where $\tilde{\cR}$ is a certain extension of $\cR$ to $\overline{Y}_{\K_f}$.

Let $A$ be a ring and $R$ be an $A$-module endowed with an action of $\mathrm{GL}_2(K)$. Let $\cR$ be the corresponding local system over $Y_{\cK_1(\fn)}$, then the Hecke algebra $\cH_{\fn}\otimes_{\Z}A$ acts on each cohomology group in (\ref{e: exact sequence}) (for more details see \cite[p.346-7]{hida1988p}, \cite[\S3,\S8]{hida1994critical}, \cite[\S 2.1]{urban1995formes}). In the same way, when $p\nmid \fn$ we have an action of the Hecke algebra $\cH_{\fn, p}\otimes_{\Z}A$ on each cohomology group appearing in (\ref{e: exact sequence}).

We are interested on certain sheaves over $\overline{Y}_{\K_f}$ corresponding to irreducible representations of $\mathrm{SU}_2(\C)$. Let $k,\ell$ be two non-negative integers. If $A$ is a $K$-algebra we denote by $V_{k,\ell}(A)$ the space of polynomials over $A$ that are homogeneous of degree $k$ in two variables $x,y$ and homogeneous of degree $\ell$ in two further variables $\overline{x}, \overline{y}$ with the $\mathrm{GL}_2(K)$-action 
\begin{equation*}
\smallmatrixx{a}{b}{c}{d}\cdot P\left[\matrixxx{x}{y},\matrixxx{\overline{x}}{\overline{y}}\right]=P\left[\matrixxx{dx+by}{cx+ay},\matrixxx{\overline{d}\overline{x}+\overline{b}\overline{y}}{\overline{c}\overline{x}+\overline{a}\overline{y}}\right].
\end{equation*}
Let $V_{k,\ell}^*(A)$ be the dual space of $V_{k,\ell}(A)$ and denote by $\cV_{k,\ell}^*(A)$ the corresponding local system on the space $\overline{Y}_{\K_f}$. 

\begin{remark}{\label{moduloberger}}
For our study of the cohomology of the Bianchi threefold we will use the works \cite{harder1984eisenstein} and \cite{berger2008denominators}, 
then we elaborate on the relation between different group actions over the polynomial spaces. First note that $V_{k,\ell}^*(\C)$ is isomorphic as a right-$\mathrm{GL}_2(K)$-module to the space of polynomials that are homogeneous of degree $k$ in two variables $s,t$ and homogeneous of degree $\ell$ in two further variables $\overline{s}, \overline{t}$ endowed with the right-$\mathrm{GL}_2(K)$-action given by 
$$ \left(P|_{\smallmatrixx{a}{b}{c}{d}}\right)\left[\matrixxx{s}{t},\matrixxx{\overline{s}}{\overline{t}}\right]=P\left[\matrixxx{ds-ct}{-bs+at},\matrixxx{\overline{d}\overline{s}-\overline{c}\overline{t}}{-\overline{b}\overline{s}+\overline{a}\overline{t}}\right].$$ Moreover, the latter space endowed with its corresponding left action of $\mathrm{GL}_2(K)$ given by $\gamma\cdot P=P|_{\gamma^{-1}}$ is isomorphic to the space $M(k,\ell,-k,-\ell)_{\C}$ from \cite[\S 2.4]{berger2008denominators}.
%On the other hand,  the latter space with the corresponding left action given by $\gamma\cdot P=P|_{\gamma^{-1}}$ is isomorphic as a left-$\mathrm{GL}_2(K)$-module to the space $M(k,\ell,-k,-\ell)_{\C}$ from \cite[\S 2.4]{berger2008denominators}.
\end{remark}

\subsection{Boundary cohomology}\label{ss: Boundary}

Let $\cK_{f}\subset \mathrm{GL}_2(\widehat{\cO}_K)$ be an open compact subgroup. We are interested in certain system of Hecke eigenvalues appearing in the cohomology groups  $H^i(\partial\overline{Y}_{\K_f},\cV_{k,\ell}^*(\C))$ for $i=0,1,2$. This groups can be completely described in terms of the spaces $V_{\chi_f}^{\K_f}$ studied in Section \ref{ss: induced representations}, explicitly:

\begin{proposition}(Harder){\label{t: boundary thm}}
We have Hecke-equivariant isomorphisms:
\begin{equation*}
H^0(\partial\overline{Y}_{\cK_{f}},\cV_{k,\ell}^*(\C))\cong \bigoplus_{\substack{\chi\;:\;T(K)\backslash T(\A_K)\rightarrow \C^\times \\ \text{of inf. type}\; [(0,0),(k,\ell)] }}V_{\chi_f}^{\cK_{f}},
\end{equation*}
\begin{equation*}
H^1(\partial\overline{Y}_{\cK_{f}},\cV_{k,\ell}^*(\C))\cong \bigoplus_{\substack{\chi\;:\;T(K)\backslash T(\A_K)\rightarrow \C^\times  \\ \text{of inf. type}\;[(k+1,0),(-1,\ell)]}} \left(V_{\chi_f}^{\cK_{f}}\oplus V_{\chi_f'}^{\cK_{f}} \right),
\end{equation*}
where $\chi':=(\chi_2|\cdot|_{\A_K},\chi_1|\cdot|_{\A_K}^{-1})$,
\begin{equation*}
H^2(\partial\overline{Y}_{\cK_{f}},\cV_{k,\ell}^*(\C))\cong \bigoplus_{\substack{\chi\;:\;T(K)\backslash T(\A_K)\rightarrow \C^\times\mathrm{of\; inf.} \\ \text{type}\; [(k+1,\ell+1)(-1,-1)]}}V_{\chi_f}^{\cK_{f}}.
\end{equation*}
\end{proposition}

\begin{remark}\label{r: Hecke equivariant} More precisely, if $\cK_{f}=\K_1(\fn)$ then the above isomorphisms are $\cH_{\fn}$-equivariant and when $\cK_{f}=\K_1(\fn, p)$ we have $\cH_{\fn,p}$-equivariant isomorphisms.
\end{remark}

\begin{proof}
This is proven by Harder in \cite[Thm 1 and \S 2.9]{harder1984eisenstein}. See also \cite[p.13, 30]{berger2008denominators} using Remark \ref{moduloberger}, for the description of $H_\partial^i$ for $i=0,1$.   
\end{proof}

\begin{proposition}\label{p: igualdadboundary}
Let $\phi_1,\phi_2$ be Hecke characters of conductors $\gn_1$, $\gn_2$ respectively, satisfying $(\gn_1,\gn_2)=1$. Write $\fn= \fn_1 \fn_2$ and let $\gm_{\phi, S}$ be the ideal introduced in equation (\ref{e: ideal}). If $\phi=(\phi_1,\phi_2)$ has infinity type $[(k+1,0),(-1,\ell)]$ (resp. $[(k+1,\ell+1), (-1,-1)]$) and $S$ has Dirichlet density greater than $1/2$, then $$\mathrm{dim}_{\C}H^0(\partial\overline{Y}_{\K_1(\gn)},\cV_{k,\ell}^*(\C))[[\gm_{\phi,S}]]=0\;(\mathrm{resp.}\;1),$$ 
$$\mathrm{dim}_{\C}H^1(\partial\overline{Y}_{\K_1(\gn)},\cV_{k,\ell}^*(\C))[[\gm_{\phi,S}]]=2\;(\mathrm{resp.}\;0),$$
$$\mathrm{dim}_{\C}H^2(\partial\overline{Y}_{\K_1(\gn)},\cV_{k,\ell}^*(\C))[[\gm_{\phi,S}]]=0\;(\mathrm{resp.}\;1).$$ 
\end{proposition}

\begin{proof}
We just prove the result on degree $0$ since degrees $1,2$ are proved analogously.

By Proposition \ref{t: boundary thm} we have,
$$H^0(\partial\overline{Y}_{\K_1(\gn)},\cV_{k,\ell}^*(\C))[[\gm_{\phi,S}]]\cong \bigoplus_{\substack{\chi\;:\;T(K)\backslash T(\A_K)\rightarrow \C^\times \\ \text{of inf. type}\; [(0,0),(k,\ell)] }}V_{(\chi_1,\chi_2)_f}^{\cK_1(\fn)}[[\gm_{\phi,S}]].$$

If $\phi=(\phi_1,\phi_2)$ has infinity type $[(k+1,0),(-1,\ell)]$, then all direct summands are $\{0\}$ by part ii) of Remark \ref{r: 0dim result}, and $H^0(\partial\overline{Y}_{\K_1(\gn)},\cV_{k,\ell}^*(\C))[[\gm_{\phi,S}]]=\{0\}$.

If $\phi=(\phi_1,\phi_2)$ has infinity type $[(k+1,\ell+1), (-1,-1)]$, note that by Theorem \ref{t: dim1borde} if some direct summand $V_{(\chi_1,\chi_2)_f}^{\cK_1(\fn)}[[\gm_{\phi,S}]]\neq \{0\}$, then either $(\chi_1,\chi_2)=(\phi_1,\phi_2)$ or $(\chi_1,\chi_2)= (\phi_2|\cdot|_{\A_K},\phi_1|\cdot|_{\A_K}^{-1})$, however the former option is not possible since the infinity types are different. 

We conclude by Theorem \ref{t: dim1borde} that 
\begin{equation*}
H^0(\partial\overline{Y}_{\K_1(\gn)},\cV_{k,\ell}^*(\C))[[\gm_{\phi,S}]]\cong V_{(\phi_2|\cdot|_{\A_K},\phi_1|\cdot|_{\A_K}^{-1})_f}^{\K_1(\gn)},
\end{equation*}
is a 1-dimensional space.
\end{proof}

\begin{remark}\label{r: dimesnion 1 borde} Observe that the weights of $\phi$ satisfy the hypothesis on the weights in our multiplicity-one Theorem \ref{t: dim1borde}. This allows us to obtain our results, which are a key ingredient in the proof of one of our main results: Theorem \ref{t: main theorem}.
\end{remark}

\begin{corollary}\label{c: igualdadboundary}
Let $\phi$ be as in Proposition \ref{p: igualdadboundary} with $p\nmid\gn$. Let $\gm_{\tilde{\phi}, S}$ be the ideal introduced in equation (\ref{e: ideal p-estabilizado}). If $\phi=(\phi_1,\phi_2)$ has infinity type $[(k+1,0),(-1,\ell)]$ (resp. $[(k+1,\ell+1), (-1,-1)]$) and $S$ has Dirichlet density greater than $1/2$, then $$\mathrm{dim}_{\C}H^0(\partial\overline{Y}_{\K_1(\gn,p)},\cV_{k,\ell}^*(\C))[[\gm_{\tilde{\phi},S}]]=0\;(\mathrm{resp}\;1),$$ 
$$\mathrm{dim}_{\C}H^1(\partial\overline{Y}_{\K_1(\gn,p)},\cV_{k,\ell}^*(\C))[[\gm_{\tilde{\phi},S}]]=2\;(\mathrm{resp}\;0),$$
$$\mathrm{dim}_{\C}H^2(\partial\overline{Y}_{\K_1(\gn,p)},\cV_{k,\ell}^*(\C))[[\gm_{\tilde{\phi},S}]]=0\;(\mathrm{resp}\;1).$$ 
\end{corollary}
\begin{proof}
The proof is exactly the same as the proof of Proposition \ref{p: igualdadboundary} using Corollary \ref{c: dim1borde} and Remark \ref{r: 0dim result p-estab} instead Theorem \ref{t: dim1borde} and part ii) of Remark \ref{r: 0dim result}.
\end{proof}

\section{Bianchi Eisenstein eigensystems}\label{s: BEE}

\subsection{Non-cuspidal base change} \label{ss: basechange}
For $\cK_f$ an open compact subgroup of $\mathrm{GL}_2(\widehat{\cO}_K)$ and $(k,\ell)\in \N\times \N$ we denote by $\cM_{(k,\ell)}(\cK_{f})$ the finite-dimensional $\C$-vector space of Bianchi modular forms of weight $(k,\ell)$ and level $\cK_{f}$. %: which are vector-valued functions on $\mathrm{GL}_2(\A_K)$ satisfying suitable transformation, harmonicity and growth conditions.
We denote by $S_{(k,\ell)}(\cK_{f})$ the subspace of cuspidal Bianchi modular forms. From \cite[Cor 2.2]{hida1994critical} we have that $S_{(k,\ell)}(\cK_{f})=\{0\}$ if $k\neq \ell$ i.e., all non-trivial cuspidal Bianchi modular forms have parallel weight $(k,k)$. Hecke operators act on the spaces $\cM_{(k,\ell)}(\K_f)$ and $S_{(k,\ell)}(\K_f)$. In particular, if $\cK_{f}=\K_1(\fn)$ then the Hecke algebra $\cH_{\fn}\otimes_{\Z}\C$ acts on those spaces. In the same way, if $\cK_{f}=\K_1(\fn, p)$ then we have an action of $\cH_{\fn,p}\otimes_{\Z}\C$ on the corresponding spaces.  %An \emph{eigenform is a simultaneous eigenvector.
For more details on Bianchi modular forms we refer the reader to \cite{chris2017}.

An example of a Bianchi modular form which is not cuspidal is obtained by base change from $\mathrm{GL}_{2, \Q}$. Let $\varphi: \A_K^{\times}\rightarrow \C^{\times}$ be a Hecke character of $K$ of infinity type $(-k-1,0)$ and conductor $\gm$. We put $M=N(\gm)$ and we denote by $\varphi_{\mathbb{Z}}: (\Z/ M\Z)^{\times}\rightarrow \C^{\times}$ the Dirichlet character induced by the map $a\mapsto \varphi(a\mathcal{O}_K)a^{-k-1}$, for $a\in \Z$ coprime to $M$. Let $f$ be the theta series associated to $\varphi$, which is a newform of weight $k+2$, level $\Gamma_0(DM)$ and nebentypus $\epsilon_{f}=\chi_K\varphi_{\mathbb{Z}}$ where $\chi_K: (\Z/D\Z)^{\times}\rightarrow \{\pm 1\}$ is the quadratic character of $K/\Q$.

Let $\pi$ be the automorphic representation of $\mathrm{GL}_2(\mathbb{A}_{\mathbb{Q}})$ generated by $f$ and let BC($\pi$) be the base change of $\pi$ to $\mathrm{GL}_2(\mathbb{A}_K)$ (see \cite{langlands1980base}). The \textit{base change} of $f$ to $K$ is the normalized new vector $\cF$ in BC($\pi$) which is a non-cuspidal Bianchi modular form of weight $(k, k)$, level $\K_0(\fn)$ with $M\mathcal{O}_K|\fn|DM\mathcal{O}_K$ (see \cite[\S 2.3]{friedberg1983imaginary}) and nebentypus $\epsilon_{\cF}=\epsilon_{f}\circ N_{K/\Q}$. 

In order to apply the main results of this work to this non-cuspidal base change case, we will suppose (compare with Remark 3.6 in \cite{BW21}): 
\begin{assumption}
The conductor, $\gm$, of $\varphi$ is coprime to $(p)$, $(\gm,\overline{\gm})=1$ and the level of $\cF$ is exactly $M\cO_K$.
\end{assumption}

The Hecke eigenvalues of $\cF$ can be described in terms of the eigenvalues $a_q$ of $f$, in particular in terms of $\varphi$: for every prime $\gq$ of $K$ above $q$,
$$
a_{\gq}=\begin{cases}
			a_q=\varphi(\gq)+\varphi(\overline{\gq}), & \text{if $q= \gq\overline{\gq}$}\\
            a_q=\varphi(\gq), & \text{if $q$ ramifies}\\
            a_q^2-2\chi_K(q)\varphi_{\mathbb{Z}}(q)q^{k+1}=2\varphi(\gq) & \text{if $q$ is inert.}
		 \end{cases}
$$

\subsection{Bianchi Eisenstein eigensystems}
%The  of the previous section is a \textit{Bianchi Eisenstein series}, and we can ask whether exists other examples.
We can ask about examples of non-cuspidal Bianchi modular forms other than the base change described in the previous section. From Harder's work we can consider Bianchi modular forms as systems of Hecke-eigenvalues appearing in the group $H^i(Y_{\K_f},\cV_{k,\ell}^*(\C))$ and we will focus on this point of view.

For each $i\in \N$, recall the Hecke-equivariant decomposition
$$H^i(Y_{\K_f},\cV_{k,\ell}^*(\C))\cong H_{cusp}^i(Y_{\K_f},\cV_{k,\ell}^*(\C)) \oplus H_{Eis}^i(Y_{\K_f},\cV_{k,\ell}^*(\C)),$$
where the cuspidal cohomology is defined by 
\begin{equation*}
H_{\mathrm{cusp}}^i(Y_{\K_f},\cV_{k,\ell}^*(\C))=\mathrm{Im}\{H_c^i(Y_{\K_f},\cV_{k,\ell}^*(\C))\rightarrow H^i(Y_{\K_f},\cV_{k,\ell}^*(\C))\},    
\end{equation*}
and its complement, the Eisenstein cohomology $H_{Eis}^i(Y_{\K_f},\cV_{k,\ell}^*(\C))$, is defined as the image of $H^i(\partial\overline{Y}_{\K_f},\cV_{k,\ell}^*(\C))$ by Harder’s Eisenstein map (see \cite{harder1984eisenstein}) 
\begin{equation*}
\mathrm{Eis}: H^i(\partial\overline{Y}_{\K_f},\cV_{k,\ell}^*(\C))\rightarrow H^i(Y_{\K_f},\cV_{k,\ell}^*(\C)).   
\end{equation*}

The following theorem describes the cuspidal and Eisenstein cohomology.

\begin{theorem}{\label{Eichler-Shimura}}(Harder)
Let $k$ and $\ell$ be two non-negative integers.

i) The following two vanishing results hold:
$$ 
\begin{array}{ll}
H_{cusp}^i(Y_{\K_f},\cV_{k,\ell}^*(\C))=0 & \textrm{unless} \;k\neq\ell \;\textrm{and either}\; i = 1 \;\mathrm{or}\; 2.   \\ 
H_{Eis}^0(Y_{\K_f},\cV_{k,\ell}^*(\C))=0 & \textrm{unless}\; k=\ell=0 \;\text{when it is equal to}\;\C. 
\end{array}
$$

ii) There are Hecke-equivariant isomorphisms.
\begin{equation*}
H_{cusp}^1(Y_{\K_f},\cV_{k,k}^*(\C))\cong H_{cusp}^2(Y_{\K_f},\cV_{k,k}^*(\C))\cong S_{k,k}(\K_f),
\end{equation*}
\begin{equation*}
H_{Eis}^1(Y_{\K_f},\cV_{k,\ell}^*(\C))\cong\bigoplus_{\substack{\chi\;:\;T(K)\backslash T(\A_K)\rightarrow \C^\times \\ \text{of inf. type}\; [(k+1,0),(-1,\ell)] }}V_{\chi_f}^{\K_f},
\end{equation*}
Moreover, if $(k,\ell)\neq(0,0)$ then 
\begin{equation*}
H_{Eis}^2(Y_{\K_f},\cV_{k,\ell}^*(\C))\cong H^2(\partial\overline{Y}_{\cK_{f}},\cV_{k,\ell}^*(\C)) \cong \bigoplus_{\substack{\chi\;:\;T(K)\backslash T(\A_K)\rightarrow \C^\times \;\text{of inf. type}\\ [(k+1,\ell+1),(-1,-1)]}}V_{\chi_f}^{\K_f}.    
\end{equation*}
%Moreover, if $k=\ell=0$, then $H_{Eis}^0(Y_{\K_f},\cV_{k,\ell}^*(\C))=\C$, and $H_{Eis}^2(Y_{\cK_1(\fn)},\cV_{k,\ell}^*(\C))$ is of codimension one in $H^2(\partial \overline{Y}_{\cK_1(\fn)},\cV_{k,\ell}^*(\C))$ .
\end{theorem}
\begin{remark} Note that as in Remark \ref{r: Hecke equivariant}, if $\cK_{f}=\K_1(\fn)$ then the above isomorphisms are $\cH_{\fn}\otimes_{\Z}\C$-equivariant and when $\cK_{f}=\K_1(\fn, p)$ we have an $\cH_{\fn,p}\otimes_{\Z}\C$-equivariant isomorphism.
\end{remark}

\begin{proof} 
For all the results except the description of $H_{Eis}^i$ for $i=1,2$, see \cite[pag. 101]{taylor1988congruences}. The result of $H_{Eis}^1$ is a consequence of \cite[p. 105]{taylor1988congruences}, Remark \ref{moduloberger} and Proposition \ref{t: boundary thm}. The result of $H_{Eis}^2$ comes from \cite[p. 101]{taylor1988congruences} combined with Proposition \ref{t: boundary thm}.
\end{proof}

By Theorem \ref{Eichler-Shimura}, the space of cuspidal Bianchi modular forms  of weight $(k,\ell)$ and level $\K_1(\gn)$ (which is defined analytically), corresponds precisely to the space of systems of Hecke eigenvalues $\cH_\gn\otimes_{\Z}\C\rightarrow\C$ appearing in $H_{cusp}^i(Y_{\K_f},\cV_{k,\ell}^*(\C))$ for some $i \in \{1,2\}$ (for each $i\in \{ 1, 2\}$, by Theorem \ref{Eichler-Shimura} part ii)). This motivates the following definition:

\begin{definition}{\label{Bianchi Eisentein}} 
A \textit{Bianchi Eisenstein eigensystem} of level $\cK_1(\gn)$ and weight $(k,\ell)$ is a system of Hecke eigenvalues appearing in $H_{Eis}^i(Y_{\cK_1(\gn)},\cV_{k,\ell}^*(\C))$ for some $i\in \{0, 1,2\}$.
\end{definition}

\begin{remark} The index $i= 0$ in the above definition only appears when $(k, \ell)= (0, 0)$. In the main results of this text we do not study this case.     
\end{remark}

%\begin{remark} In this text we study the eigenvariety around points coming from $p$-stabilizations of Bianchi Eisenstein eigensystems of level $\cK_1(\fn)$ attached to certain characters $\phi= (\phi_1, \phi_2)$. {\color{green}(D: precisar $\phi$)} 
%\end{remark}

Proposition \ref{p: igualdadboundary} allows us to state the following result about Bianchi Eisenstein eigensystems.
 
\begin{corollary}\label{c: newformseisenstein} Let $\phi_1,\phi_2$ be Hecke characters of conductors $\gn_1$, $\gn_2$ respectively, satisfying $(\gn_1,\gn_2)=1$. Write $\fn= \fn_1 \fn_2$ and let $\gm_{\phi, S}$ the ideal introduced in equation (\ref{e: ideal}) with $S$ of Dirichlet density greater than $1/2$. If $\phi=(\phi_1,\phi_2)$ has infinity type $[(k+1,0),(-1,\ell)]$ (resp. $[(k+1,\ell+1), (-1,-1)]$ with $(k, \ell)\neq (0,0)$) then
$$\mathrm{dim}_{\C}H_{Eis}^1(Y_{\K_1(\gn)},\cV_{k,\ell}^*(\C))[[\gm_{\phi,S}]]=1\;(\mathrm{resp.}\;0),$$
$$\mathrm{dim}_{\C}H_{Eis}^2(Y_{\K_1(\gn)},\cV_{k,\ell}^*(\C))[[\gm_{\phi,S}]]=0\;(\mathrm{resp.}\;1).$$
\end{corollary}

\begin{proof}
We start with the results in degree 1. 
Suppose $\phi$ has infinity type $[(k+1,0),(-1,\ell)]$, then by part ii) of Theorem \ref{Eichler-Shimura} together with Theorem \ref{t: dim1borde} we have that
$$H_{Eis}^1(Y_{\K_1(\gn)},\cV_{k,\ell}^*(\C))[[\gm_{\phi,S}]]\cong\bigoplus_{\substack{\chi\;:\;T(K)\backslash T(\A_K)\rightarrow \C^\times, \text{ of}\\ \text{inf. type}\; [(k+1,0),(-1,\ell)] }}V_{\chi_f}^{\K_1(\fn)}[[\gm_{\phi,S}]]=V_{(\phi_1,\phi_2)_f}^{\K_1(\fn)},$$
is a 1-dimensional $\C$-vector space. 

Consider now $\phi$ of infinity type $[(k+1,\ell+1), (-1,-1)]$ with $(k, \ell)\neq (0,0)$, then by part ii) of Remark \ref{r: 0dim result} we obtain $H_{Eis}^1(Y_{\K_1(\gn)},\cV_{k,\ell}^*(\C))[[\gm_{\phi,S}]]=\{0\}$.

We now proceed with the results on degree 2. Suppose first $(k,\ell)\neq (0, 0)$, then we have
$$H_{Eis}^2(Y_{\K_1(\gn)},\cV_{k,\ell}^*(\C))\cong H^2(\partial\overline{Y}_{\K_1(\gn)},\cV_{k,\ell}^*(\C)),$$
(see \cite{taylor1988congruences}), and by Proposition \ref{p: igualdadboundary} we have the result.

For the remaining case in degree $2$ where $\phi$ has weight $[(1,0),(-1,0)]$, we obtain the result since $\mathrm{dim}_{\C}H^2(\partial Y_{\K_1(\gn)},\cV_{k,\ell}^*(\C))[[\gm_{\phi,S}]]=0.$ 
\end{proof}

Note that in particular, the previous corollary gives us that each $\phi=(\phi_1,\phi_2)$ determines a Bianchi Eisenstein eigensystem of level $\cK_1(\gn)$ and weight $(k,\ell)$.

Recall that we write $\gm_{\phi}$ instead of $\gm_{\phi, S}$ when $S=\{\gq \text{ prime }: \gq\nmid\gn\}$. 

\begin{theorem}{\label{t: main theorem}}
Let $k,\ell$ be two non-negative integers, let $\phi_1,\phi_2$ be two Hecke characters of conductors $\gn_1$, $\gn_2$ such that $(\gn_1,\gn_2)=1$, and write $\fn= \fn_1 \fn_2$.

i) If $(k,\ell)\neq(0,0)$ then the cohomology groups $H_{?}^i(Y_{\K_f},\cV_{k,\ell}^*(\C))$ for $?\in\{\emptyset, c\}$ are concentrated in degree 1 and 2.

ii) If $\phi=(\phi_1,\phi_2)$ has infinity type $[(k+1,0),(-1,\ell)]$ (resp. $[(k+1,\ell+1), (-1,-1)]$ with $(k, \ell)\neq (0, 0)$) then: 
$$\mathrm{dim}_{\C}H^1(Y_{\K_1(\gn)},\cV_{k,\ell}^*(\C))[[\gm_\phi]]=\mathrm{dim}_{\C}H_c^2(Y_{\K_1(\gn)},\cV_{k,\ell}^*(\C))[[\gm_\phi]]=1 \;(\mathrm{resp.}\; 0),$$ 
$$
\mathrm{dim}_{\C}H^2(Y_{\K_1(\gn)},\cV_{k,\ell}^*(\C))[[\gm_\phi]]=\mathrm{dim}_{\C}H_c^1(Y_{\K_1(\gn)},\cV_{k,\ell}^*(\C))[[\gm_\phi]]=0 \;(\mathrm{resp.}\; 1).$$
\end{theorem}
\begin{proof}
For (i) see the proof of \cite[Lem 4.3.1]{lee2023p}. 

We now prove ii) when $(k,\ell)\neq(0,0)$. To ease notation we write $H^i$ instead of $H^i(Y_{\K_1(\gn)},\cV_{k,\ell}^*(\C))$, and the same for the type of cohomologies. The results for $H^i$ for $i=1,2$ in part ii) follows from 
\begin{equation}\label{e: betti}
H^i[[\gm_\phi]]=H_{cusp}^i[[\gm_\phi]]\oplus H_{\mathrm{Eis}}^i[[\gm_\phi]]=H_{\mathrm{Eis}}^i[[\gm_\phi]], 
\end{equation}
and Corollary \ref{c: newformseisenstein}.

For the results in compact support, we consider the following exact sequence
\begin{equation}\label{exact}
0\rightarrow H_\partial^0\rightarrow H_c^1\rightarrow H^1\rightarrow H_\partial^1\rightarrow H_c^2\rightarrow H^2\rightarrow H_\partial^2\rightarrow 0.   
\end{equation}
Now we take the generalized eigenspace corresponding to $\gm_\phi$ in (\ref{exact}), and after using the previous results of $H^j$ for $j=1,2$ and the results for $H_\partial^j$ for $j=0,1,2$ of Proposition \ref{p: igualdadboundary}, we conclude by dimensional reasons and the exactness of the sequence.

We now proceed with the result when $(k,\ell)=(0,0)$ in ii), i.e., the character $\phi$ has infinity type $[(1,0),(-1,0)]$. First, we note that the proof of \cite[Lem 4.3.1]{lee2023p} also give us that $H^3(Y_{\K_1(\gn)},\C)=H_c^0(Y_{\K_1(\gn)},\C)=\{0\}$. Now, consider the exact sequences 
\begin{align}\label{e: caso (0,0)}
   & 0\rightarrow H_c^0[[\gm_\phi]]\rightarrow H^0[[\gm_\phi]]\rightarrow H_\partial^0[[\gm_\phi]]\rightarrow \cdots \\
   & \nonumber\cdots\rightarrow H_\partial^2[[\gm_\phi]]\rightarrow H_c^3[[\gm_\phi]]\rightarrow H^3[[\gm_\phi]]\rightarrow \cdots,
\end{align}
and note that by Proposition \ref{p: igualdadboundary} we have $H_\partial^0[[\gm_\phi]]=H_\partial^2[[\gm_\phi]]=0$, and in particular, by the previous argument $H_c^0[[\gm_\phi]]=H^3[[\gm_\phi]]=0$, then we conclude that $H^0[[\gm_\phi]]=H_c^3[[\gm_\phi]]=0$.

The result for $H^j$ for $j=1,2$ follows by (\ref{e: betti}) and Corollary \ref{c: newformseisenstein}, and the result for $H_c^j$ for $j=1,2$ follows in the same way as before, by taking the generalized eigenspace in (\ref{exact}) and using the previous results about $H^j$ for $j=1,2$ and $H_\partial^j$ for $j=0,1,2$.
\end{proof}

\begin{remark}
\begin{enumerate}
    \item Note that by setting $\phi_1=\varphi^{-1}$ and $\phi_2=(\varphi^c)^{-1}|\cdot|_{\A_K}^{-1}$ in the above theorem, we recover the eigensystem determined by the base change Bianchi modular form associated to $\varphi$ (see Section \ref{ss: basechange}).    
    \item We can prove a strong multiplicity-one version of the theorem above by invoking the main result of  \cite{ramakrishnan1994refinement} when studying the cuspidal part of the cohomology. However, the lower bound on the density of $S$ should be $7/8$ instead of $1/2$, condition inherited from \cite{ramakrishnan1994refinement}.

\end{enumerate}
\end{remark}

\subsection{\texorpdfstring{$p$}{Lg}-stabilized eigensystems} \label{ss: $p$-stabilized eigensystems} 
Points corresponding to Bianchi modular forms in the Bianchi eigenvariety are attached through the so called $p$-\emph{stabilizations}. We precise them first for the base change Bianchi modular form $\cF$ attached to a Hecke character $\varphi$ from Section \ref{ss: basechange}, followed by Bianchi Eisenstein eigensystems. 

We suppose that the conductor of $\varphi$, $\fm$, is coprime to $p$, then $\cF$ has level coprime to $p$. Some of the $p$-stabilizations of $\cF$ can be linked to the $p$-stabilizations of the theta series associated to $\varphi$,
$f$, as follows.

The Hecke polynomial of $f$ at $p$ is given by $x^2-a_px+ \epsilon_{f}(p)p^{k+1}$. Since
$a_p=\varphi(\gp)+\varphi(\overline{\gp})$ and $\epsilon_{f}(p)=\chi_K(p)\varphi_{\mathbb{Z}}(p)=\varphi(p\mathcal{O}_K)/p^{k+1}=\varphi(\gp\overline{\gp})/p^{k+1}$,
the roots of this polynomial are $\alpha=\varphi(\overline{\gp})$ and $\beta=\varphi(\gp)$. 

Likewise, the Hecke polynomial of $\cF$ at $\gq|p$ is given by $x^2-a_{\gq}x+ \epsilon_{\cF}(\gq)N(\gq)^{k+1}$, and since
$a_\gq=a_p$ (see Section \ref{ss: basechange}), $N(\gq)=p$, $\epsilon_{\cF}(\gq)=\epsilon_{f}(p)$, then the Hecke polynomials of $\cF$ at $\gp$ and $\overline{\gp}$ and the Hecke polynomial of $f$ at $p$ are equal. Thus, we obtain four $p$-stabilizations attached to $\cF$ denoted by $\cF^{\alpha\alpha}, \cF^{\beta\beta}, \cF^{\alpha\beta}$ and  $\cF^{\beta\alpha}$ (see for example \cite{palacios}). The form $\cF^{\alpha\alpha}$ has $U_{\fq}$-eigenvalue given by $\alpha$ for each $\fq\mid p$ and similarly for $\cF^{\beta\beta}, \cF^{\alpha\beta}$ and  $\cF^{\beta\alpha}$.

Denote by $f^{\alpha}$ and $f^{\beta}$ the $p$-stabilizations of $f$ corresponding to $\alpha$ and $\beta$ respectively. Then they correspond to the $p$-stabilizations $\cF^{\alpha\alpha}$ and $\cF^{\beta\beta}$ of $\cF$  respectively. This fact, have important consequences on the geometry of the Bianchi eigenvariety around points corresponding to $\cF^{\alpha\alpha}$ and $\cF^{\beta\beta}$ (see Section \ref{ss: results eigenvariety}). 

%\begin{remark}{\label{basechangepstabilization}}
%Note that $v_p(\alpha)=v_p(\varphi(\overline{\gp}))=0$ and $v_p(\beta)=v_p(\varphi(\gp))=k+1$, then $\cF^{\alpha\alpha}$ is ordinary and $\cF^{\beta\beta}$ is critical.
%\end{remark}

In the case of Bianchi Eisenstein eigensystems we can perform an analogous process of the $p$-stabilization. In fact, let $\phi=(\phi_1,\phi_2)$ be a character as in Corollary \ref{c: newformseisenstein}, which gives us a Bianchi Eisenstein eigensystem of level $\K_1(\gn)$ and weight $(k,\ell)$, and suppose additionally that $p\nmid\gn$. For $\gq|p$ recall the Hecke polynomial (\ref{e: hecke polynomial}) and its roots $\alpha_\fq$ and $\beta_\fq$.

\begin{definition}\label{d:stabilized BEE}
A \emph{$p$-stabilized Bianchi Eisenstein eigensystem} of $\phi$ is a couple $\tilde{\phi}= (\phi, (x_{\fq})_{\fq\mid p}))$ where $x_{\fq}\in \{\alpha_{\fq}, \beta_{\fq}\}$ for each $\fq\mid p$. 
\end{definition} 
Recall the ideal  $\gm_{\tilde{\phi}}$ from (\ref{e: ideal p-estabilizado}) with $S$ the full set of primes coprime to $p\fn$.
\begin{corollary}\label{c: newformeisenstein p-estab}
Let $\phi$ be as in Corollary \ref{c: newformseisenstein} with $p\nmid\gn$ and infinity type $[(k+1,0),(-1,\ell)]$ (resp. $[(k+1,\ell+1), (-1,-1)]$ with $(k, \ell)\neq (0, 0)$). Then:
$$\mathrm{dim}_{\C} H_{Eis}^1(Y_{\K_1(\gn,p)},\cV_{k,\ell}^*(\C))[[\gm_{\tilde{\phi}}]]=1,\;(\mathrm{resp.}\;0) ,$$
$$\mathrm{dim}_{\C} H_{Eis}^2(Y_{\K_1(\gn,p)},\cV_{k,\ell}^*(\C))[[\gm_{\tilde{\phi}}]]=0,\;(\mathrm{resp.}\;1).$$
\end{corollary}
\begin{proof}
The proof is exactly the same as the proof of Corollary \ref{c: newformseisenstein} using Corollary \ref{c: dim1borde} and Remark \ref{r: 0dim result p-estab} instead Theorem \ref{t: dim1borde} and part ii) of Remark \ref{r: 0dim result}.
%Consider first degree 1, and note that for infinity type $[(k+1,0),(-1,\ell)]$ we use Theorem \ref{Eichler-Shimura} and Corollary \ref{c: dim1borde} with $S=\{\gq \text{ prime }: \gq\nmid p\gn\}$ and proceed as in the proof of Corollary \ref{c: newformseisenstein}. For infinity type $[(k+1,\ell+1), (-1,-1)]\neq[(1,1),(-1,-1)]$ we just use Remark \ref{r: 0dim result p-estab}. The results on degree 2 follows from the fact that for $(k,\ell)\neq (0, 0)$  we have $H_{Eis}^2(Y_{\K_1(\gn,p)},\cV_{k,\ell}^*(\C))\cong H^2(\partial\overline{Y}_{\K_1(\gn,p)},\cV_{k,\ell}^*(\C)),$ combined with Corollary \ref{c: igualdadboundary}.
\end{proof}

This corollary shows that $\tilde{\phi}$ is in fact a Bianchi Eisenstein eigensystem of level $\K_1(\gn,p)$ and weight $(k,\ell)$.

\begin{remark} \label{r: valuaciones general}
Note that if $\phi$ has infinity type $[(k+1,0),(-1,\ell)]$ (resp. $[(k+1,\ell+1), (-1,-1)]$), then $v_p(\alpha_\gp)=v_p(\beta_{\overline{\gp}})=0, v_p(\alpha_{\overline{\gp}})=\ell+1$, $v_p(\beta_\gp)=k+1$ (resp.  $v_p(\alpha_\gp)=v_p(\alpha_{\overline{\gp}})=0$, $v_p(\beta_\gp)=k+1$,  $v_p(\beta_{\overline{\gp}})=\ell+1$). Since we will work with the ordinary $p$-stabilization, we are interested in $\tilde{\phi}= (\phi, \alpha_\gp, \beta_{\overline{\gp}})$ (resp. $\tilde{\phi}= (\phi, \alpha_\gp, \alpha_{\overline{\gp}})$).
\end{remark}

%\begin{remark}{\label{algebraicity}}
%So far we have been working with complex coefficients, however, the vanishing results in Theorem \ref{t: main theorem} follows from $\cV_{k,\ell}^*(\overline{\Q})$, as well as the 1-dimensional ones, since the map defined by Harder from the spaces $V_{\phi_f}^{\K_f}$ to the singular cohomology is rational, see \cite[\S 3.6]{berger2008denominators}.
%\end{remark}

\begin{corollary}\label{c: dimension 1 de p-estabilizaciones} Let $\phi$ be as in Theorem \ref{t: main theorem}, fix $\tilde{\phi}$ a $p$-stabilization of $\phi$ and let $L$ be a finite extension of $\Q_p$ containing the values of $\tilde{\phi}$. We put $\fn= \fn_1\fn_2$ and denote by $\fm_{\tilde{\phi}}\subset \cH_{\fn, p}\otimes_{\Z}L$ the maximal ideal attached to $\tilde{\phi}$ as in equation (\ref{e: ideal p-estabilizado}). Then if $\phi$ has infinity type $[(k+1,0),(-1,\ell)]$ (resp. $[(k+1,\ell+1), (-1,-1)]$ and $(k,\ell)\neq (0, 0)$) then
$$\mathrm{dim}_{L}H^1(Y_{\K_1(\gn, p)},\cV_{k,\ell}^*(L))_{\fm_{\tilde{\phi}}}=\mathrm{dim}_{L}H_c^2(Y_{\K_1(\gn, p)},\cV_{k,\ell}^*(L))_{\fm_{\tilde{\phi}}}=1 \text{ (resp. $0$)},$$ 
$$\mathrm{dim}_{L}H^2(Y_{\K_1(\gn, p)},\cV_{k,\ell}^*(L))_{\fm_{\tilde{\phi}}}=\mathrm{dim}_{L}H_c^1(Y_{\K_1(\gn, p)},\cV_{k,\ell}^*(L))_{\fm_{\tilde{\phi}}}= 0 \text{ (resp. $1$)}.$$
\end{corollary}
\begin{proof} Firstly, remark that from \cite[Thm 2.5.9]{bellaiche2021eigenbook} in each statement of the corollary we can work the generalized eigenspaces instead of localizations. 

Fixing a big enough number field $F$, for each $i$ the space $H^i(Y_{\K_1(\gn, p)},\cV_{k,\ell}^*(F)[[\fm_{\tilde{\phi}}]]$ is a $F$-rational structure of $H^i(Y_{\K_1(\gn, p)},\cV_{k,\ell}^*(L)[[\fm_{\tilde{\phi}}]]$ and $H^i(Y_{\K_1(\gn, p)},\cV_{k,\ell}^*(\C)[[\fm_{\tilde{\phi}}]]$. The same is true for compact support coholomogy groups. Thus, it is enough to prove the statement of the corollary with $\C$ instead of $L$. Now, to prove this we follow the strategy of the proof of Theorem \ref{t: main theorem} using Corollary \ref{c: newformeisenstein p-estab} and equations 
(\ref{e: betti}), (\ref{exact}) and (\ref{e: caso (0,0)}).
\end{proof}

\begin{remark}\label{r: recoverBC p-estab}
Note that by setting $\phi_1=\varphi^{-1}$ and $\phi_2=(\varphi^c)^{-1}|\cdot|_{\A_K}^{-1}$ in the above corollary, we recover the system of eigenvalues determined by the $p$-stabilizations of the base change Bianchi modular form associated to $\varphi$ (see Section \ref{ss: $p$-stabilized eigensystems}).    
\end{remark}
 
In \cite{palacios2023} were introduced and studied \textit{partial modular symbols} to produce $p$-adic $L$-functions for the ordinary $p$-stabilization ${\cF}^{\alpha\alpha}$. In a forthcoming work, the second author describes these partial modular symbols in terms of cohomology and perform the same study carried out in the current section in order to deform these $p$-adic $L$-functions. Following \cite{palacios2023} and \cite{bellaiche2015p} such study would produce interesting arithmetic applications such as the construction of $4$-variables $p$-adic $L$-functions passing through the one attached to $\cF^{\alpha\alpha}$.

\section{Bianchi eigenvarieties}

Using the multiplicity-one results for Bianchi Eisenstein eigensystems proved in Section \ref{s: BEE} we study Bianchi eigenvarieties around non-cuspidal points attached to these eigensystems.

\subsection{Overconvergent modules} We consider Bianchi eigenvarieties constructed in \cite{hansen}. Also see \cite{BW21} for a presentation of Hansen's work in the Bianchi setting.  

We fix $\fn$ an ideal of $\cO_K$ coprime with $p$.

\begin{definition} The \emph{weight space} is the rigid analytic space $\mathcal{W}$ defined over $\Q_p$ whose $R$-points, for a $\Q_p$-affinoid algebra $R$, are
\begin{equation*}
\mathcal{W}(R) = \mathrm{Hom}_{cts}(\tensorspace^\times/E(\gn), R^\times).  
\end{equation*}
where $E(\gn):=\{\epsilon \in \roi_K^{\times}: \epsilon\equiv1\; \mathrm{mod}\;\gn\}$. 
\end{definition} 

As $p$ splits in $K$ then we have $\cO_{K, p}^{\times}\cong \Z_p^{\times}\times \Z_p^{\times}$ and using this we will denote by $(z, \overline{z})$  the elements of $\cO_{K, p}^{\times}$. If $L$ is a finite extension of $\Q_p$  then a weight $\lambda\in\mathcal{W}(L)$ is called an \textit{algebraic weight} if $\lambda(z, \overline{z})=z^k\overline{z}^{\ell}$ with $k, \ell\in\mathbb{Z}$.

If $\cU\subset \cW$ is an admissible $L$-affinoid, we denote by $\cO(\cU)$ the ring of its rigid analytic functions. If $\lambda\in \cU(L)$ the rigid localization is denoted by $\cO(\cU)_{\lambda}$ and if $M$ is a $\cO(\cU)$-module we put $M_{\lambda}= M\otimes_{\cO(\cU)}\cO(\cU)_{\lambda}$. When $\lambda$ is algebraic and attached to $(k, \ell)\in \Z^2$ then we simply write $M_{(k, \ell)}$.

If $R$ is any $\Q_p$-affinoid algebra we denote by $\cA(R)$ the space of locally analytic functions $h: \cO_{K, p}\rightarrow R$, this space is naturally a Fréchet $R$-module. We denote by $\cD(R)$ its continuous dual which is a compact Fréchet $R$-module.

We consider the monoid
\begin{equation*}
\Sigma_0(p) := \left\{\matrixx{a}{b}{c}{d} \in M_2(\Z_p)\cap\mathrm{GL}_2(\Q_p):  c\in p\Z_p, d\in\Z_p^\times\right\}. 
\end{equation*}

For an affinoid $\cU\subset \cW$ we denote by $\lambda_{\cU}:\tensorspace^\times\rightarrow \roi(\cU)^\times$ its attached character. We denote by $\cA_\cU$ the $\cO(\cU)$-module $\cA(\roi(\cU))$ endowed with a left action by $\Sigma_0(p)\times \Sigma_0(p)$ given by $$(\gamma\cdot h)(z, \overline{z}) = \lambda_\cU(a+cz, \overline{a}+\overline{c} \overline{z}) h\left(\frac{dz+b}{cz+a},\frac{\overline{d} \overline{z} +\overline{b}}{\overline{c}\overline{z}+ \overline{a}}\right),$$ 
for $\gamma= \left(\smallmatrixx{a}{b}{c}{d}, \smallmatrixx{\overline a}{\overline b}{\overline c}{\overline d}\right)\in \Sigma_0(p)\times \Sigma_0(p)$ and $h\in \cA(\roi(\cU))$. Then we denote by $\cD_{\cU}$ the $\cO(\cU)$-module $\cD(\cO(\cU))$ endowed with the right action by $\Sigma_0(p)\times \Sigma_0(p)$ obtained from its left action on $\cA(\roi(\cU))$. When $\cU= \{\lambda\}$ is a single weight $\lambda: \cO_{K, p}^{\times}\rightarrow L^{\times}$ for $L$ a finite extension of $\Q_p$, the $L$-vector spaces $\cA_{\{\lambda\}}$ and $\cD_{\{\lambda\}}$ are simply denoted by $\cA_{\lambda}$ and $\cD_{\lambda}$. Finally, if $\lambda$ is an algebraic weight attached to the integers $k, \ell\in \Z$ then we denote the spaces obtained by $\mathcal{A}_{k,\ell}$ and $\cD_{k,\ell}$.

Let $k, \ell$ be non-negative integers and $L$ be a finite extension of $\Q_p$, then we have an $\Sigma_0(p)\times \Sigma_0(p)$-equivariant injection 
$$V_{k,\ell}(L)\longrightarrow \mathcal{A}_{k,\ell}, \ \ P\left[\matrixxx{x}{y},\matrixxx{\overline{x}}{\overline{y}}\right]\longrightarrow h(z, \overline{z}):= P\left[\matrixxx{z}{1},\matrixxx{\overline{z}}{\overline{1}}\right].$$
From this we obtain a $\Sigma_0(p)\times \Sigma_0(p)$-equivariant morphism of $L$-vector spaces
 \begin{equation}\label{e:specialization espacios}
 \cD_{k,\ell}\rightarrow V_{k,\ell}^{\ast}(L).
 \end{equation}

\subsection{Eigenvarieties}\label{ss: eigenvarieties}

Remark that we can identify $E(\fn)$ with $Z(\mathrm{GL}_2(K))\cap \cK_1(\fn, p)$ where $Z(\mathrm{GL}_2(K))= K^\times$ is the center of $\mathrm{GL}_2(K)$, thus the invariance under $E(\fn)$ allows to obtain from the spaces $\cD_\cU, \cD_{\lambda}$ and $\cD_{k,\ell}$ local systems on the manifold $Y_{\cK_1(\fn, p)}$ which are denoted again by the same symbols. We will use the cohomology groups $H^i(Y_{\cK_1(\fn, p)},\cD_{k,\ell})$ and  $H_c^i(Y_{\cK_1(\fn, p)},\cD_{k,\ell})$ for $i\in \N$. These spaces are endowed with an action of the Hecke algebra $\cH_{\fn, p}$.

For each $i\in \N$ and $?\in\{\emptyset,c\}$, the map (\ref{e:specialization espacios}) induces a Hecke-equivariant \textit{specialization map} of $L$-vector spaces
\begin{equation}\label{e:specialization map}
\rho_{k,\ell}:H_?^i(Y_{\cK_1(\fn, p)},\cD_{k,\ell})\rightarrow H_?^i(Y_{\cK_1(\fn, p)},\cV_{k,\ell}^*(L)).  
\end{equation} 
We recall the following result which controls this morphism in small enough slope subspaces (see \cite{BW21}, \cite{BW21parabolic} and \cite{hansen}). If $M$ is a $L[U_{\fp}, U_{\overline{\fp}}]$-module and  $h= (h_{\fp}, h_{\overline\fp})\in \R^2$ we denote by $M^{\leqslant h}\subset M$ the $L$-vector space of the vectors of slope $\leqslant h_{\fp}$ with respect to $U_{\fp}$ and slope $\leqslant h_{\overline\fp}$ with respect to $U_{\overline\fp}$ (see \cite[Prop. 4.6.2]{ash2008p}).

\begin{theorem}\label{t:control} If $h= (h_{\fp}, h_{\overline{\fp}})\in [0, k+1)\times [0, \ell+1)$ then (\ref{e:specialization map}) induces a Hecke equivariant isomorphism of $L$-vector spaces
$$\rho_{k,\ell}:H_?^i(Y_{\cK_1(\fn, p)},\cD_{k,\ell})^{\leqslant h}\cong H_?^i(Y_{\cK_1(\fn, p)},\cV_{k,\ell}^*(L))^{\leqslant h}$$
for $i\in \N$ and $?\in\{\emptyset,c\}$.
\end{theorem}

Firstly, we consider the eigenvariety of tame level $\cK_1(\fn)$ constructed using the full cohomology and denoted by $\cE$. Moreover, we consider the eigenvariety constructed using the compact supported cohomology and denoted by $\cE_c$. Both are endowed with finite maps to the weight space:
$$\cE\longrightarrow \cW \ \ \ \ \ \  \text{and} \ \ \ \ \ \ \cE_c\longrightarrow \cW.$$

See \cite[\S3.2]{BW21} and \cite[\S4.3]{hansen} for details. For example we give an idea of the construction of the eigenvariety $\cE_c$. Let $\cU\subset \cW$ be an open affinoid and $h\in \R$ non-negative such that $H_c^{\bullet}(Y_{\cK_1(\fn, p)}, \cD_{\cU})$ admits a slope-$\leqslant h$ decomposition with respect to $U_p= U_{\fp}U_{\overline{\fp}}$ (such a pair $(\cU, h)$ is called a slope adapted pair). We denote by 
$$\bT_{\cU, h}\subset \mathrm{End}_{\cO(\cU)}(H_c^{\bullet}(Y_{\cK_1(\fn, p)}, \cD_{\cU})^{\leqslant h})$$ 
the $\cO(\cU)$-algebra generated by the action of the Hecke algebra $\cH_{\gn,p}$ on $H^{\bullet}_c(Y_{\cK_1(\fn, p)}, \cD_{\cU})^{\leqslant h}$. The eigenvariety $\cE_c$ is constructed by gluing together the rigid analytic spectrums $\mathrm{Sp}(\bT_{\cU, h})$. The natural maps $\mathrm{Sp}(\bT_{\cU, h})\rightarrow \cU$ also glue to a map $\cE_c\longrightarrow \cW$.

\subsection{Results} \label{ss: results eigenvariety}

Let $\phi= (\phi_1, \phi_2)$ be as in Section \ref{ss: $p$-stabilized eigensystems} i.e., $\phi_1,\phi_2$ are Hecke characters of $K$ of conductors $\fn_1,\fn_2$ respectively such that $(\fn_1, \fn_2)= 1$ and $\fn= \fn_1\fn_2$. Moreover, we assume that either:

\begin{itemize}
\item $\phi$ has infinity type $[(k+1,0),(-1,\ell)]$ or
\item $\phi$ has infinity type $[(k+1,\ell+1), (-1,-1)]$ and $(k,\ell)\neq (0, 0)$.
\end{itemize}

We fix a $p$-stabilization $\tilde{\phi}$ of $\phi$, as defined in Section \ref{ss: $p$-stabilized eigensystems}. It is natural to ask if $\tilde\phi$ determines a point in some of the eigenvarieties $\cE$ or $\cE_c$. In affirmative cases then we would like to study the local geometry of the eigenvariety around these points. The following theorem treats these questions when $\tilde{\phi}$ is ordinary i.e., $v_p(x_{\fp})= v_{p}(x_{\overline{\fp}})= 0$ (see Section \ref{ss: $p$-stabilized eigensystems}).

\begin{theorem} \label{t: main general theorem} If $\tilde\phi$ is the ordinary $p$-stabilization of $\phi$, then we have:
\begin{enumerate}
    \item $\tilde{\phi}$ determines a point $x\in \cE$ and the weight map $\cE\longrightarrow \cW$ is etale at $x$.
    \item $\tilde{\phi}$ determines a point $x_{c}\in \cE_c$ and the weight map $\cE_c\longrightarrow \cW$ is etale at $x_c$.
\end{enumerate}
\end{theorem}
\begin{proof} The proof of (1) and (2) are completely analogue, thus we prove the theorem for the cohomology without support. Moreover, we will suppose that $\phi$ has infinity type $[(k+1,0),(-1,\ell)]$. The case when $\phi$ has infinity type $[(k+1,\ell+1), (-1,-1)]$ with $(k,\ell)\neq (0, 0)$ is proved in the same way. 

Recall the maximal ideal $\fm_{\tilde{\phi}}\subset \cH_{\fn, p}\otimes_{\Z}L$ attached to $\tilde{\phi}$ as in (\ref{e: ideal p-estabilizado}). As $\tilde\phi$ is the ordinary $p$-stabilization of $\phi$ then after localizing the map (\ref{e:specialization map}) at $\fm_{\tilde{\phi}}$, we obtain from Theorem \ref{t:control}  an isomorphism of $L$-vector spaces
 $$\rho_{k,\ell}:H^i(Y_{\cK_1(\fn, p)},\cD_{k,\ell})_{\fm_{\tilde{\phi}}}\rightarrow H^i(Y_{\cK_1(\fn, p)},\cV_{k,\ell}^*(L))_{\fm_{\tilde{\phi}}}$$  
for each $i\in \N$. We deduce from Theorem \ref{t: main theorem} and Corollary \ref{c: dimension 1 de p-estabilizaciones} that $H^{\bullet}(Y_{\K_1(\gn, p)},\cD_{k,\ell})_{\fm_{\tilde{\phi}}}$ is concentrated in degree $\bullet= 1$ and $\mathrm{dim}_{L}H^1(Y_{\K_1(\gn, p)},\cD_{k,\ell})_{\fm_{\tilde{\phi}}}= 1$. From this, we obtain a point $x\in \cE$ attached to $\tilde\phi$.

We fix a slope adapted pair $(\cU, h)$ and we consider the localization $(\bT_{\cU,h})_{x}$ of $\bT_{\cU, h}$ at $x$. Applying \cite[Lemmas 2.9(i) and 2.11]{BDJ17} we deduce that $H^{\bullet}(Y_{\K_1(\gn, p)},\cD_{\cU})^{\leqslant h}_{\fm_{\tilde{\phi}}}$ is supported in degree $1$ and $H^{1}(Y_{\K_1(\gn, p)},\cD_{\cU})^{\leqslant h}_{\fm_{\tilde{\phi}}}$ is a $\cO(\cU)_{(k, \ell)}$-module free of rank $1$. Then we observe that 
$$(\bT_{\cU, h})_{x}\subset \mathrm{End}_{\cO(\cU)}(H^{\bullet}(Y_{\K_1(\gn, p)},\cD_{k,\ell})^{\leqslant h}_{\fm_{\tilde{\phi}}})\cong \cO(\cU)_{(k, \ell)}$$
As $(\bT_{\cU, h})_{x}$ is a $\cO(\cU)_{(k, \ell)}$-algebra containing $1$, we deduce that $(\bT_{\cU, h})_{x}\cong \cO(\cU)_{(k, \ell)}$ and thus the map $\cE\rightarrow \cW$ is etale at $x$. 
\end{proof}

\begin{remark} When $p$ is inert in $K$ and the conductors of $\phi_1$ and $\phi_2$ are coprime to $p$ then the $p$-stabilizations of $\phi= (\phi_1, \phi_2)$ have critical slope and then we can not use the corresponding control Theorem \ref{t:control}. Since that theorem is key in the proof of the previous theorem, this explains our assumption on the splitting behavior of $p$ when the conductors of $\phi$ are coprime to $p$.
\end{remark}

\begin{corollary}\label{c: rank 1 around our point} There exists a slope adapted pair $(\cU,h)$ such that $(k,\ell)\in \cU$ and if $\cC \subset \cE_c$ denotes the connected component though $x_c$ over $\cU$ we have 
\begin{enumerate}
\item if $\phi$ has infinity type $[(k+1,0),(-1,\ell)]$ then $H^2_c(Y_{\K_1(\gn, p)}, \cD_{\cU})^{\leqslant h}\otimes_{\bT_{\cU}}\cO(\cC)$ is a $\cO(\cU)$-module free of rank $1$;
\item if $\phi$ has infinity type $[(k+1,\ell+1), (-1,-1)]$ with $(k,\ell)\neq (0, 0)$ then the $\cO(\cU)$-module $H^1_c(Y_{\K_1(\gn, p)}, \cD_{\cU})^{\leqslant h}\otimes_{\bT_{\cU}}\cO(\cC)$ is free of rank $1$. 
\end{enumerate}
\end{corollary}
\begin{proof} We suppose that $\phi$ has infinity type $[(k+1,0),(-1,\ell)]$ as the proof is identical in the other case. As in the proof of Theorem \ref{t: main general theorem}, from Theorem \ref{t: main theorem} and Corollary \ref{c: dimension 1 de p-estabilizaciones} we deduce that there exists a slope adapted pair $(\cU, h)$ such that $H^{\bullet}_c(Y_{\K_1(\gn, p)},\cD_{\cU})^{\leqslant h}_{\fm_{\tilde{\phi}}}$ is supported in degree $2$ and $H^{2}_c(Y_{\K_1(\gn, p)},\cD_{\cU})^{\leqslant h}_{\fm_{\tilde{\phi}}}$ is a $\cO(\cU)_{(k, \ell)}$-module free of rank $1$. Finally, we apply \cite[Lemma 2.10]{BDJ17} to deduce that shrinking $\cU$ we can delocalize this, and thus proving the corollary. 
\end{proof}

\begin{remark} \begin{enumerate} \item These results are proved under the hypothesis that the conductor of $\phi_1$ is coprime to the conductor of $\phi_2$. Such condition comes from the Theorem \ref{t: dim1borde}. To relax these hypotheses, we should consider a slightly different open compact subgroups and a modification of the new vectors considered in the proof of Lemma \ref{l: eigen} (for more details in the modifications see \cite[\S 3.2]{berger2008denominators}). 
        \item Taking $\phi_1$ and $\phi_2$ as in Remark \ref{r: recoverBC p-estab}, then the theorem and corollary above proved the Corollary \ref{c: BC intro} for the base change points. Moreover, observe that in this case we allow $(k, \ell)= (0,0)$, thus in particular we consider eigenvalues attached to  base change elliptic curves from $\Q$ having complex multiplication by $K$.
\end{enumerate}
\end{remark}

Consider the base change Bianchi modular from $\mathcal{F}$ from Section \ref{ss: basechange} and its four $p$-stabilizations $\cF^{\alpha\alpha}$, $\cF^{\alpha\beta}$, $\cF^{\beta\alpha}$, $\cF^{\beta\beta}$ from Section \ref{ss: $p$-stabilized eigensystems}. The theorem above proved that around $\cF^{\alpha\alpha}$ the Bianchi eigenvariety is etale, then it is natural to ask what is happening around the other $p$-stabilizations. For example, if any of these $p$-stabilizations determine a point in the eigenvariety and in the affirmative case how it is the geometry around such a point. The $p$-stabilization $\cF^{\beta\beta}$ is special: by \cite{bellaiche2012critical} the Coleman-Mazur eigencurve is smooth around the modular form $f_{\beta}$ from Section \ref{ss: $p$-stabilized eigensystems}, then taking the base change of the Coleman family passing through $f_{\beta}$ (which is cuspidal and not CM) we obtain a curve $\cC_{BC}\subset \cE$ and thus $\cF^{\beta\beta}$ determines a point $x_\beta \in \cE$. We thank C. Williams for pointing out that $\cC_{BC}$ is not contained in any $2$-dimensional family (we apply \cite[Thm. 3.8]{BW21} to any non-critical point in $\cC_{BC}$). As observed by A. Betina, since the slope is locally constant in families we should be able to prove that there is no a $2$-dimensional (even a classical) non-cuspidal family passing through $x_{\beta}$. More precisely, we can prove: 

\begin{proposition} If $\cI$ is an irreducible component of $\cE$ such that $x_\beta \in \cI$ and contains infinitely many classical points around $x_\beta$, then all but finitely many of these points are cuspidal.
\end{proposition}
\begin{proof}(Sketch) Suppose infinitely many of these points are non-cuspidal points, then such points should be non-ordinary $p$-stabilizations of classical eigensystems. But those $p$-stabilizations have slope given by their weights (see remark \ref{r: valuaciones general}). We obtain a contradiction because the slope is locally constant over $p$-adic families.
\end{proof}

\begin{question}
Are there classical families other than $\cC_{BC}$ in $\cE$ and passing through $x_{\beta}$?    
\end{question} 

It is interesting to compare this question with the results from \cite{BDS23} by Betina, Dimitrov and Shih in the real quadratic case. Moreover, the authors of this work would be very interested in study what happens with $\cF^{\alpha\beta}$ and $\cF^{\beta\alpha}$. Finally, we observe that in \cite{eisensteindegeneration} the authors study the geometry of the Coleman-Mazur eigencurve around critical slope Eisenstein series based on Bellaïche's work. Their results suggest that $\cF^{\beta\beta}$ would appear in degrees $1$ and $2$ of the overconvergent cohomology. 

\subsection{\texorpdfstring{$p$}{Lg}-adic \texorpdfstring{$L$}{Lg}-functions of Bianchi Eisenstein eigensystems}\label{sss: $p$-adic $L$-functions} Our previous results allow the construction of $p$-adic $L$-functions attached to certain Bianchi Eisenstein eigensystems. Let $\tilde\phi$ be the ordinary $p$-stabilization of $\phi= (\phi_1, \phi_2)$ as in Section \ref{ss: results eigenvariety} and suppose that its infinity type is $[(k+1,\ell+1), (-1,-1)]$ with $(k,\ell)\neq (0, 0)$.

From Theorem \ref{t:control} and Corollary \ref{c: dimension 1 de p-estabilizaciones} we deduce that for big enough finite extension $L$ of $\Q_p$, the $L$-vector space
$H^1_c(Y_{\cK_1(\fn, p)},\cD_{k,\ell})_{\fm_{\tilde\phi}}$ is one dimensional. We pick any non-zero class $\Psi_{\tilde\phi}\in H^1_c(Y_{\cK_1(\fn, p)},\cD_{k,\ell})_{\fm_{\tilde\phi}}$. 

From \cite[\S2.4]{BW21} we have the Mellin transform
$$\mathrm{Mel}_{(k, \ell)}: H^1_c(Y_{\cK_1(\fn, p)},\cD_{k,\ell})\rightarrow \cD(\mathrm{Cl}_K(p^{\infty}), L),$$
where $\mathrm{Cl}_K(p^{\infty})=K^\times\backslash \mathbb{A}_K^\times/\mathbb{C}^\times\prod_{\gq\nmid p}\cO_\gq^\times$.

\begin{definition}
The \emph{$p$-adic $L$-function of $\tilde\phi$} is the locally analytic distribution
 $$L_p(\tilde{\phi}):= \mathrm{Mel}_{(k, \ell)}(\Psi_{\tilde\phi})\in \cD(\mathrm{Cl}_K(p^{\infty}), L).$$
\end{definition}

\begin{proposition} \label{p: deformation of L-function} The distribution $L_p(\tilde{\phi})$ is bounded i.e., it is a measure over the $p$-adic group $\mathrm{Cl}_K(p^{\infty})$. Moreover,   $L_p(\tilde{\phi})$ can be deformed in a $2$-dimensional family of measures over $\mathrm{Cl}_K(p^{\infty})$. 
\end{proposition}
\begin{proof} By exactly the same proof given in \cite[Prop. 6.15]{chris2017}  we obtain that the distribution $L_p(\tilde{\phi})$ is $(0,0)$-admissible. This is equivalent to say that $L_p(\tilde{\phi})$ is bounded i.e. it is a measure.

Now let $(\cU,h)$ be a slope adapted pair as in Corollary \ref{c: rank 1 around our point}. Then if $\cC\subset \mathrm{Sp}(\bT_{\cU, h})$ is the connected component passing through the point corresponding to $\tilde{\phi}$ then $H^1_c(Y_{\K_1(\gn, p)}, \cD_{\cU})^{\leqslant h}\otimes_{\bT_{\cU}}\cO(\cC)$ is a free $\cO(\cU)$-module of rank $1$. We pick a class 
$$\Psi_{\cU}\in H^1_c(Y_{\K_1(\gn, p)}, \cD_{\cU})^{\leqslant h}\otimes_{\bT_{\cU}}\cO(\cC)\subset H^1_c(Y_{\K_1(\gn, p)}, \cD_{\cU})^{\leqslant h}$$
lifting $\Psi_{\tilde\phi}$. 

Recall from \cite[\S2.4]{BW21} we have the Mellin transform
$$\mathrm{Mel}_{\cU}: H^1_c(Y_{\cK_1(\fn, p)},\cD_{\cU})\rightarrow \cD(\mathrm{Cl}_K(p^{\infty}), \cO(\cU)).$$ 
Thus, the distribution $L_p(\cC):= \mathrm{Mel}_{\cU}(\Psi_{\cU})\in \cD(\mathrm{Cl}_K(p^{\infty}), \cO(\cU))$ is a $2$-dimensional family of distributions over $\mathrm{Cl}_K(p^{\infty})$ specializing to $L_p(\tilde\phi)$ at $(k, \ell)$. 
\end{proof}

\begin{remark} We expect that $L_p(\tilde{\phi})$ satisfies an interpolation formula as a \emph{true} $p$-adic $L$-function. In fact, we would expect it is related to special values of Hecke $L$-functions. 
\end{remark}

\begin{remark}
Recall that our assumption on the splitting behavior of $p$ is motivated by the need of having  ordinary $p$-stabilizations and thus to be able to apply the control Theorem \ref{t:control}. However, when $p$ is inert and if we suppose the character $\phi=(\phi_1,\phi_2)$ has infinity type $[(k+1,\ell+1), (-1,-1)]$ and the conductors of $\phi_1,\phi_2$ satisfy $p||\gn_1$ and $p\nmid\gn_2$ then $\phi$ is attached to an ordinary Eisenstein eigensystem. If we also suppose $(k,\ell)\neq(0,0)$ then following the argument in Theorem \ref{t: main general theorem} we would obtain etaleness of the corresponding eigenvarieties. In the same way as in this section we can attach a $p$-adic $L$-function to $\phi$ and deform it in a $2$-dimensional family of measures over $\mathrm{Cl}_K(p^\infty)$ (exactly in the same way as in Proposition \ref{p: deformation of L-function}). It would very interesting to obtain an interpolation formula relating these to complex $L$-values and study the existence of trivial zeros. 
\end{remark}

\bibliographystyle{amsalpha}
\bibliography{main}
\end{document}